\documentclass[letterpaper, 11pt]{amsart}
\usepackage{amsmath,amssymb,amsthm,pinlabel,tikz,hyperref,mathrsfs,color, thmtools}
\usepackage{thm-restate}
\DeclareSymbolFontAlphabet{\amsmathbb}{AMSb}%
\usepackage{verbatim}
\usepackage{float}
\usepackage{caption}
\usepackage{subcaption}
\usepackage{enumitem}
\newcommand{\nc}{\newcommand}
\nc{\dmo}{\DeclareMathOperator}
\dmo{\ra}{\rightarrow}
\dmo{\Prob}{\mathbb{P}}
\dmo{\E}{\mathbb{E}}
\dmo{\N}{\mathbb{N}}
\dmo{\Z}{\mathbb{Z}}
\dmo{\Q}{\mathbb{Q}}
\dmo{\R}{\mathbb{R}}
\dmo{\C}{\mathcal{C}}
\dmo{\X}{\mathcal{X}}
\dmo{\U}{\mathcal{U}}
\dmo{\T}{\mathcal{T}}
\dmo{\F}{\mathcal{F}}
\dmo{\AC}{\mathcal{AC}}
\dmo{\w}{\omega}
\dmo{\MIN}{\mathcal{MIN}}
\dmo{\Mod}{Mod}
\dmo{\PMod}{PMod}
\dmo{\PMF}{\mathcal{PMF}}
\dmo{\Mat}{Mat}
\dmo{\supp}{supp}
\dmo{\UE}{\mathcal{UE}}
\dmo{\vol}{vol}
\dmo{\B}{B}
\dmo{\PB}{PB}
\dmo{\PR}{PSL(2,\mathbb{R})}
\dmo{\GL}{GL(k, \mathbb{C})}
\dmo{\SL}{SL(2, \mathbb{Z})}
\dmo{\Isom}{Isom}
\dmo{\RP}{\mathbb{R} \mathrm{P}}
\dmo{\I}{\mathcal{I}}
\dmo{\el}{\ell_{\C}}
\dmo{\NN}{\mathcal{N}}
\dmo{\rk}{rank}
\dmo{\tr}{tr}
\dmo{\llangle}{\langle\langle}
\dmo{\rrangle}{\rangle\rangle}
\dmo{\Unif}{Unif}
\dmo{\Out}{Out}
\dmo{\diam}{\operatorname{diam}}
\dmo{\Aut}{\operatorname{Aut}}
\dmo{\sumRho}{\mathscr{B}}
\dmo{\stopping}{\vartheta}
\dmo{\diffPivot}{\mathcal{P}}
\dmo{\diffEvenPivot}{\mathcal{Q}}
\dmo{\varGam}{\Upsilon}
\dmo{\prodSeq}{\Pi}
\dmo{\NSupp}{N_{supp}}
\dmo{\KSleep}{\mathnormal{K_{sleep}}}
\dmo{\Devi}{\upsilon}
\dmo{\DeviStr}{\iota}
\dmo{\DeviDisc}{\varrho}
\dmo{\sphere}{\mathcal{S}}
\dmo{\DeviUni}{\varsigma}
\dmo{\wVar}{\check{\w}}
\dmo{\muVar}{\eta}
\dmo{\sublinear}{\Delta}
\dmo{\axes}{\mathbf{Y}}

\usetikzlibrary{decorations.markings, patterns}
\usetikzlibrary{decorations.pathmorphing}
\usetikzlibrary{calc, arrows, positioning}
\tikzset{->-/.style={decoration={
  markings,
  mark=at position #1 with {\arrow{>}}},postaction={decorate}}}

\nc{\nt}{\newtheorem}

\nt{theorem}{Theorem}

\newtheorem{thm}{{\bf Theorem}}[section]
\newtheorem{conv}[thm]{{\bf Convention}}
\newtheorem{lem}[thm]{{\bf Lemma}}
\newtheorem{cor}[thm]{{\bf Corollary}}
\newtheorem{prop}[thm]{{\bf Proposition}}

\newtheorem{remark}[thm]{Remark}

\newtheorem{definition}[thm]{Definition}

\numberwithin{equation}{section}
\newtheorem{obs}[thm]{Observation}

\title[Random walks and contracting elements III]{Random walks and contracting elements III: Outer space and outer automorphism group}

\date{\today}
\author{Inhyeok Choi}

\address{%
		Cornell University \\
		310 Malott Hall, Ithaca, NY, USA, 14850\newline
		June E Huh Center for Mathematical Challenges, KIAS\\
		85 Hoegiro Dongdaemun-gu, Seoul 02455, Republic of Korea
}
\email{%
	inhyeokchoi48@gmaill.com
        }

\begin{document}

\begin{abstract}
Continuing from \cite{choi2022random1}, we study random walks on (possibly asymmetric) metric spaces using the bounded geodesic image property (BGIP) of certain isometries. As an application, we show that a generic outer automorphism of the free group of rank at least 3 has different forward and backward expansion factors. This answers a question of Handel and Mosher in \cite{handel2007parageometric}. Together with this, we also revisit limit laws on Outer space including SLLN, CLT, LDP and the genericity of a triangular fully irreducible outer automorphisms.


\noindent{\bf MSC classes:} 20F67, 30F60, 57M60, 60G50
\end{abstract}

\maketitle

\section{Introduction}\label{section:intro}

This is the third in a series of articles concerning random walks on metric spaces with contracting elements. This series is a reformulation of the previous preprint \cite{choi2022limit} announced by the author. In this article, we adopt the following convention:

Let $X$ be a space with possibly an asymmetric metric. We say that a subset $A$ of $X$ exhibits \emph{bounded geodesic image property} (\emph{BGIP} in short) if the closest point projection of a geodesic $\gamma$ onto $A$ is uniformly bounded whenever $\gamma$ is far away from $A$. An isometry $g$ of $X$ has \emph{BGIP} if the orbit $\{g^{n} o\}_{n \in \Z}$ is a bi-quasigeodesic with BGIP (see Definition \ref{dfn:boundedGeod}). Two BGIP isometries $g$ and $h$ of $X$ are \emph{independent} if there orbits have unbounded Hausdorff distance.

\begin{conv}\label{conv:main}
Throughout, we assume that: \begin{itemize}
\item $(X, d)$ is a geodesic metric space, possibly with an asymmetric metric;
\item $G$ is a countable group of isometries of $X$, and
\item $G$ contains two  independent BGIP isometries.
\end{itemize}
We also fix a basepoint $o \in X$.
\end{conv}

The main purpose of this article is to generalize the random walk theory in \cite{choi2022random1} to asymmetric metric spaces. Our main result concerns the mismatch of the forward translation length $\tau(g)$ and the backward translation length $\tau(g^{-1})$ of a generic isometry $g$ of $X$.

\begin{theorem}[Asymmetry of a generic translation length]\label{thm:mismatch}
Let $(X, G)$ be as in Convention \ref{conv:main}. Let $(Z_{n})_{n\ge 0}$ be the random walk generated by a non-elementary, asymptotically asymmetric measure $\mu$ on $G$. Then for any $K>0$, we have \[
\lim_{n\rightarrow \infty} \Prob\Big( \big|\tau(Z_{n}) - \tau(Z_{n}^{-1})\big| < K\Big) =0.
\]
\end{theorem}

As an application, we obtain the following corollary:

\begin{cor}\label{cor:mismatchOuter}
Let $X$ be Culler-Vogtmann Outer space of rank $N \ge 3$ and $G$ be the corresponding outer automorphism group of the free group $F_{N}$. Let $(Z_{n})_{n \ge 0}$ be an admissible random walk on $G$. Then for any $K>0$, we have  \[
\lim_{n\rightarrow \infty} \Prob\Big( \w : \big|\tau(Z_{n}) - \tau(Z_{n}^{-1})\big| < K\Big) =0.
\]
\end{cor}

Corollary \ref{cor:mismatchOuter} asserts that a generic outer automorphism has different expansion factor than its inverse. This was suggested by Handel and Mosher in \cite{handel2007parageometric}. There, they proved the asymmetry for a large class of automorphisms, namely, the class of parageometric fully irreducibles (\cite[Theorem 1]{handel2007parageometric}). Nonetheless, Corollary \ref{cor:mismatchOuter} does not follow from Handel and Mosher's result, as a generic outer automorphism is ageometric and not parageometric (\cite[Corollary C]{kapovich2022random}). Meanwhile, a companion result in \cite[Theorem 1.1]{handel2007the-expansion} provides a constant $C$ depending on the rank $N$ of the free group such that, for any fully irreducible element $\phi \in \Out(F_{N})$, the expansion factor $\lambda$ of $\phi$ and $\lambda'$ of $\phi^{-1}$ satisfy $1/C \le \log (\lambda/\lambda') \le C$. In other words, the translation lengths of $\phi$ and $\phi^{-1}$ are within bounded ratio (see also \cite[Theorem 23]{algom-kfir2012asymmetry}).

Together with this, we also recover Horbez's SLLN \cite[Theorem 5.7]{horbez2016horo} and CLT  \cite[Theorem 0.2]{horbez2018clt} for displacement, and Dahmani-Horbez's SLLN for translation length \cite[Theorem 1.4]{dahmani2018spectral} on Outer space, the latter with a weaker and optimal moment condition. We also present a CLT for translation length and its converse, which seems new for Outer space. Moreover, we obtain optimal deviation inequalities on Outer space; see \cite{horbez2018clt} for previously known deviation inequalities. Using them, we also establish the geodesic tracking of random walks. Finally, we also discuss the exponential genericity of (atoroidal) fully irreducible automorphisms, which is a recurring theme in \cite{maher2018random}, \cite{taylor2016random} and \cite{kapovich2022random}; note that we do not require moment conditions here.

In order to apply our general theory to Outer space, we crucially utilize the BGIP of fully irreducible outer automorphisms. Namely, we modify Kapovich-Maher-Pfaff-Taylor's observation (\cite[Theorem 7.8]{kapovich2022random}) into the following form: 
 
\begin{restatable}{thm}{iwipBGIP}\label{thm:iwipBGIP}
Let $N\ge 3$. Then every fully irreducible outer automorphism in $\Out(F_{N})$ has BGIP with respect to the closest point projection.
\end{restatable}

\subsection{Structure of the article} \label{subsect:str}

In Section \ref{section:prelim}, we recall the notion of bounded geodesic image property (BGIP) and prove relevant lemmata. These lemmata were used in \cite{choi2022random1} to establish the alignment of sample paths of a random walk, which we rephrase in the language of BGIP. In Section \ref{section:limit}, we summarize and generalize the limit laws discussed in \cite{choi2022random1} and \cite{choi2022random2} while pointing out a subtle difference. In Section \ref{section:outer}, we review facts about the outer automorphism group and Outer space. The main result of this section is Theorem \ref{thm:iwipBGIP}, the BGIP of fully irreducible outer automorphisms. Combining the results in Section \ref{section:limit} and Section \ref{section:outer}, we obtain limit laws for $\Out(F_{N})$ in Section \ref{section:limitOuter}.

\subsection*{Acknowledgments}
The author thanks Hyungryul Baik, Talia Fern{\'o}s, Ilya Gekhtman, Thomas Haettel,  Joseph Maher, Hidetoshi Masai, Catherine Pfaff, Yulan Qing, Kasra Rafi,  Samuel Taylor and Giulio Tiozzo for helpful discussions. In particular, the author is grateful to Samuel Taylor for providing the author with a simpler proof of Theorem \ref{thm:iwipBGIP}. The author is also grateful to the American Institute of Mathematics and the organizers and the participants of the workshop ``Random walks beyond hyperbolic groups'' in April 2022 for helpful and inspiring discussions.

The author was partially supported by Samsung Science \& Technology Foundation grant No. SSTF-BA1702-01. This work constitutes a part of the author's PhD thesis.

\section{Preliminaries} \label{section:prelim}
In this section, we gather necessary notions and facts about geodesics, paths, and bounded geodesic image property (BGIP).

\subsection{Asymmetric metric spaces}\label{subsect:metric}

\begin{definition}[Metric space]\label{dfn:metric}
An \emph{(asymmetric) metric space} $(X, d)$ is a set $X$ equipped with a function $d : X \times X \rightarrow \mathbb{R}_{\ge 0}$ that satisfies the following: \begin{itemize}
\item (non-degeneracy) for each $x, y \in X$, $d(x, y) = 0$ if and only if $x = y$;
\item (triangle inequality) for each $x, y, z \in X$, $d(x, z) \le d(x, y) + d(y, z)$;
\item (local symmetry) for each $x \in X$, there exist $\epsilon, K>0$ such that $d(y, z) \le K d(z, y)$ holds for $y, z \in \left\{a \in X : \min\left(d(x, a), d(a, x)\right)<\epsilon\right\}$.
\end{itemize}
In this situation, we say that $d$ is a \emph{metric} on $X$. $d$ is said to be \emph{symmetric} if $d(x, y) = d(y, x)$ holds for all $x, y \in X$. We define a symmetric metric called the \emph{symmetrization} of $d$ by \[
d^{sym}(x, y) := d(x, y) + d(y, x).
\]
We endow $(X, d)$ with the topology induced by $d^{sym}$.
\end{definition}

We define the Gromov product between $x$ and $y$ based at $z$ by \[
(x, y)_{z} := \frac{1}{2} [d(x, z) + d(z, y) - d(x, y)].
\]

From now on, we fix an (asymmetric) metric space $(X, d)$. The \emph{diameter} of a set $A \subseteq X$ is defined by \[
\diam(A) := \sup \{d(x, y) : x, y \in A\},
\] and the \emph{(directed) distances} between sets $A, B \subseteq X$ are defined by \[\begin{aligned}
d(A, B) &:= \inf \{d(x, y) : x \in A, y \in B\}, \\
d^{sym}(A, B) &:=  \inf\{d^{sym}(x, y) : x \in A, y \in B\}.
\end{aligned}
\]
For $R>0$, the \emph{$R$-neighborhood} of a set $A\subseteq X$ is defined by \[
\mathcal{N}_{R}(A) :=\{x : d^{sym}(x, A) < R\}.
\] The \emph{Hausdorff distance} between $A, B \subseteq X$ is defined by \[
d_{H}(A, B):=\inf\{R >0 : A \subseteq \mathcal{N}_{R}(B) \,\,\textrm{and}\,\, B \subseteq \mathcal{N}_{R}(A)\}.
\]

Given $A, B \subseteq X$, we say that $A$ is \emph{$K$-coarsely contained by} $B$ if $A \subseteq \mathcal{N}_{K}(B)$, and that $A$ and $B$ are \emph{$K$-coarsely equivalent} if $d_{H}(A, B) < K$. We say that $A$ is \emph{$K$-coarsely connected} if for each $x, y \in A$ there exists a a chain $x = a_{0}, a_{1} \ldots, a_{n} = y$ of points in $A$ such that $d^{sym}(a_{i}, a_{i+1}) < K$. 

A \emph{path} is a map from an interval $I$ or a set $I$ of consecutive integers into $X$. Given a path $\gamma : I \rightarrow X$ and $a < b$ in $\mathbb{R}$, we denote its restriction to $I \cap [a, b]$ by $\gamma|_{[a, b]}$. Also, we define the \emph{reversal} $\bar{\gamma}$ of $\gamma$ by reversing the orientation, i.e., by precomposing the map $t \mapsto -t$ from $-I$ to $I$.

We say that two paths $\gamma : I \rightarrow X$, $\eta : J \rightarrow X$ are \emph{$K$-fellow traveling} if there exist a non-decreasing surjective map $\varphi : I \rightarrow J$ such that $d^{sym}(\gamma(t), (\eta \circ \varphi) (t)) < K$ for each $t \in I$.

\begin{definition}[Quasigeodesics]\label{dfn:quasigeod}
A path $\gamma : I \rightarrow X$ from an interval or a set of consecutive integers $I$ is called a \emph{$K$-quasigeodesic} if \begin{equation}\label{eqn:quasiGeod}
\frac{1}{K} |t-s| - K \le d(\gamma(s), \gamma(t)) \le K|t-s| + K
\end{equation}
holds for all $s, t \in I$ such that $s<t$. If Inequality \ref{eqn:quasiGeod} holds for all $s, t \in I$, we say that $\gamma$ is a \emph{$K$-bi-quasigeodesic}. An $1$-quasigeodesic from an interval to $X$ is called a \emph{geodesic}.

A metric space $X$ is said to be \emph{geodesic} if every ordered pair of points can be connected by a geodesic, i.e.,  for every $x, y \in X$ there exists a geodesic $\gamma : [a, b] \rightarrow X$ such that $\gamma(a) = x$ and $\gamma(b) = y$. We denote $\gamma$ by $[x, y]$.
\end{definition}

\begin{remark}\label{rem:asymGeod}
In many asymmetric metric spaces (including Culler-Vogtmann Outer space), the reversal of a geodesic may not be a geodesic. Meanwhile, geodesics on asymmetric metric spaces are continuous thanks to the local symmetry of the metric. 
\end{remark}

We will frequently use Inequality \ref{eqn:quasiGeod} in the following form. For any points $p, q$ on a $K$-bi-quasigeodesic $\gamma$, we have \begin{equation}\label{eqn:biQuasiRev1}
\diam\left(\gamma^{-1}(p) \cup \gamma^{-1}(q)\right) \le Kd(p, q) + K^{2},
\end{equation} \begin{equation}\label{eqn:biQuasiRev2}
d(q, p) \le K\diam\left(\gamma^{-1}(p) \cup \gamma^{-1}(q)\right) + K \le K^{2}d(p, q) + K^{3} + K.
\end{equation}

Fixing a basepoint $o \in X$, we define the \emph{translation length} of an isometry $g$ of $X$ by \[
\tau(g) := \lim_{n \rightarrow \infty} \frac{1}{n} d(o, g^{n} o).
\]
The triangle inequality tells us that $\tau(g)$ does not depend on the choice of $o$. Meanwhile, because the metric $d$ can be asymmetric, $\tau(g)$ and $\tau^{-1}(g)$ are not equal in general.

\subsection{Bounded geodesic image property (BGIP)} \label{subsect:BGIP}

In this subsection, we fix a (possibly asymmetric) metric space $(X, d)$. Given a subset $A$ of $X$, we define the closest point projection $\pi_{A} : X \rightarrow 2^{A}$ onto $A$ by \[
\pi_{A}(x) := \{ a \in A : d(x, a) = d(x, A)\}.
\]

\begin{definition}[Bounded geodesic image property]\label{dfn:boundedGeod}
A subset $A \subseteq X$ of a geodesic metric space $X$ is said to satisfy the \emph{$K$-bounded geodesic image property}, or $K$-BGIP in short, if the following hold: \begin{enumerate}
\item for each $z \in X$, $\pi_{A}(z)$ is nonempty;
\item for each geodesic $\eta$ such that $d^{sym}(\eta, A) > K$, we have $\diam(\pi_{A}(\eta) ) \le K$.
\end{enumerate}

A $K$-bi-quasigeodesic that satisfies $K$-BGIP is called a \emph{$K$-BGIP axis}.
\end{definition}

\begin{obs}\label{obs:BGIPReversal}
The reversal of a $K$-BGIP axis is again a $K$-BGIP axis.
\end{obs}

Note that strongly contracting property (as in \cite[Definition 2.1]{choi2022random1}) and BGIP are not equivalent on asymmetric metric spaces. Hence, we need additional arguments to establish the analogues of results in \cite{choi2022random1}. Roughly speaking, we need to replace $d(x, y)$, the distance between two points $x$ and $y$, with its symmetrization $d^{sym}(x, y)$ in the arguments in \cite{choi2022random1} and \cite{choi2022random2}. For example, we now define that: 

\begin{definition}[{\cite[Definition 5.8]{bestvina2009higher}}]\label{dfn:indep}
Bi-infinite paths $\kappa = (x_{i})_{i \in \Z}$, $\eta = (y_{i})_{i \in \Z}$ are said to be \emph{independent} if the map $(n, m) \mapsto d^{sym}(x_{n}, y_{m})$ is proper, i.e., for any $M > 0$, $\{(n, m) : d^{sym}(x_{n}, y_{m}) < M\}$ is bounded.

Isometries $g, h$ of $X$ are said to be \emph{independent} if their orbits are independent.
\end{definition}

\begin{definition}\label{dfn:nonElt}
A subgroup of $\Isom(X)$ is said to be \emph{non-elementary} if it contains two independent BGIP isometries.
\end{definition}

From now on, we always discuss under Convention \ref{conv:main}: $(X, d)$ is a geodesic metric space, with possibly asymmetric metric, $G$ is a countable group of isometries of $X$, and $G$ contains two independent BGIP isometries.

\subsection{Random walks} \label{subsect:RW}

We recall the notations in \cite[Subsection 2.4]{choi2022random1}.

Let $\mu$ be a probability measure $G$, which comes with its reflected version $\check{\mu}$ defined by $\check{\mu}(g) := \mu(g^{-1})$. The \emph{random walk} generated by $\mu$ is the Markov chain on $G$ with the transition probability $p(g, h) := \mu(g^{-1} h)$; this can be defined on a probability space on which the \emph{step elements} $\{g_{n}\}_{n \in \Z}$ are measurable, where $g_{i}$'s are i.i.d.s distributed according to $\mu$. Such a probability space is called the \emph{probability space for $\mu$}. Equivalently,  $(\Omega, \Prob)$ is a probability space for $\mu$ if there is a measure preserving map from $(\Omega, \Prob)$ to $(G^{\Z}, \mu^{\Z})$.

Given a step path $(g_{n})_{n \in \Z} \in (G^{\Z}, \mu^{\Z})$, we define the \emph{(bi-infinite) sample path} $(Z_{n})_{n \in \Z}$ by \[
Z_{n} = \left\{ \begin{array}{cc} g_{1} \cdots g_{n} & n > 0 \\ id & n=0 \\ g_{0}^{-1} \cdots g_{n+1}^{-1} & n < 0. \end{array}\right.
\]
We also introduce the notation $\check{g}_{n} = g_{-n+1}^{-1}$ and $\check{Z}_{n} = Z_{-n}$. Note that we have an isomorphism $(G^{\Z}, \mu^{\Z}) \rightarrow (G^{\Z_{>0}}, \check{\mu}^{\Z_{>0}}) \times (G^{\Z_{>0}}, \mu^{\Z_{>0}})$ by $(g_{n})_{n \in \Z} \mapsto ((\check{g}_{n})_{n >0} , (g_{n})_{n > 0})$. In view of this, we sometimes write the bi-infinite sample path as $((\check{Z}_{n})_{n\ge0}, (Z_{n})_{n\ge0})$.

When a constant $M_{0}$ is understood, we will denote the $M_{0}$-long subpath of the sample path ending at $Z_{i} o$ by $\mathbf{Y}_{i}$ In other words, we denote the sequence $\big(Z_{i-M_{0}}(\w) o, Z_{i - M_{0} + 1}(\w) o, \ldots, Z_{i-1}(\w) o, Z_{i} (\w) o \big)$ by  $\mathbf{Y}_{i}(\w)$.

We define the \emph{support} of $\mu$, denoted by $\supp \mu$, as the set of elements of $G$ that are assigned nonzero value by $\mu$. We denote by $\mu^{N}$ the product measure of $N$ copies of $\mu$, and by $\mu^{\ast N}$ the $N$-th convolution of $\mu$.

A probability measure $\mu$ on $G$ is said to be \emph{non-elementary} if the semigroup $\llangle \supp \mu \rrangle$ generated by the support of $\mu$ contains two independent BGIP isometries $g, h$ of $X$. $\mu$ is said to be \emph{admissible} if $\llangle \supp \mu \rrangle$ equals the entire group $G$.  $\mu$ is said to be \emph{non-arithmetic} if there exist $N > 0$ and $g, h \in \supp \mu^{\ast N}$ such that $\tau(g) \neq \tau(h)$. Finally, we say that $\mu$ is \emph{asymptotically asymmetric} if there exists $N>0$ and $g, h \in \supp \mu^{\ast N}$ such that \[
\tau(g) - \tau(g^{-1}) \neq \tau(h) - \tau(h^{-1}).
\] The random walk $(Z_{n})_{n \ge 0}$ generated by $\mu$ is said to be admissible (non-elementary, non-arithmetic or asymptotically asymmetric, resp.) if $\mu$ is admissible (non-elementary, non-arithmetic or asymptotically asymmetric, resp.).

For a given $p>0$, we define the $p$-th moment of $\mu$ by \[
\E_{\mu}[d(o, go)^{p}] := \sum_{g \in G} d(o, go)^{p}\, \mu(g).
\]
Note that $\E_{\mu}[d(o, go)^{p}]$ and \[
\E_{\check{\mu}}[d(o, go)^{p}] := \sum_{g\in G} d(o, go)^{p} \,\check{\mu}(g) = \sum_{g \in G} d(go, o)^{p} \,\mu(g)
\]
are \emph{distinct in general}, and the finitude of the former does not imply that of the latter. This technicality leads to a subtle difference between limit laws for symmetric and asymmetric metric spaces. However, many asymmetric metric spaces (including Outer space) satisfy the following coarse symmetry: there exists a global constant $K>0$ such that $d(x, y) \le K d(x, y)$ for $x, y \in Go$. Under such a coarse symmetry, a measure $\mu$ has finite $p$-th moment if and only if its reflected version $\check{\mu}(\cdot) := \mu(\cdot^{-1})$ does so. Hence, this subtlety will not matter for Outer space and many other spaces.

\subsection{Properties of BGIP axes} \label{subsect:BGIPProperty}
In \cite{choi2022random1}, we summarized some useful properties of contracting axes that were established earlier by many authors in \cite{arzhantseva2015growth}, \cite{sisto2018contracting}, \cite{yang2019statistically}. Here, we record the analogous properties for BGIP axes.

\begin{lem}[cf. {\cite[Lemma 3.1]{choi2022random1}}]\label{lem:coarseEquiv}
Let $K>1$, let $\gamma$ be a $K$-BGIP axis and let $\eta : I \rightarrow X$ be a geodesic such that $\diam (\pi_{\gamma}(\eta)) > K$. Then there exist $t < t'$ in $I$ such that $\pi_{\gamma}(\eta)$ and $\eta|_{[t, t']}$ are $20K^{3}$-coarsely equivalent, and moreover, \[
\diam \big(\pi_{\gamma}(\eta|_{(-\infty, t]}) \cup \eta(t) \big) < 7K^{3}, \quad \diam \big(\pi_{\gamma}( \eta|_{[t', +\infty)}) \cup \eta(t')\big) < 7K^{3}.
\]
\end{lem}

This was proved for symmetric metric spaces in \cite[Lemma 2.14]{arzhantseva2015growth}, \cite[Lemma 2.4]{sisto2018contracting} and \cite[Lemma 2.4(4)]{yang2019statistically}. We give a proof for asymmetric metric spaces, which is an adaptation of \cite[Lemma 2.2]{chawla2023genericity}, for completeness.

\begin{proof}
Let $N = N_{K}(\gamma)$ be the $K$-$d^{sym}$-neighborhood of $\gamma$, $\bar{N}$ be its closure and $S = \eta \cap \bar{N}$. Then $S$ is closed. Moreover, since $\diam(\pi_{\gamma}(\eta)) > K$, $S$ is nonempty by the $K$-BGIP of $\gamma$. We now claim:

\begin{obs}\label{obs:coarseEquivObs1}
Let $x \in S$, let $y \in \gamma$ and suppose that $d^{sym}(x, y) \le K$. Then $d(y, \pi_{\gamma}(x)) \le K$ and $d(\pi_{\gamma}(x), y) \le 2K^{3} + K$.
\end{obs}

\begin{proof}[Proof of Observation \ref{obs:coarseEquivObs1}]
We have \[
d\big(y, \pi_{\gamma}(x) \big) \le d\big(y, x \big)  + d\big(x, \pi_{\gamma}(x) \big) \le d\big(y, x \big)  + d\big(x, y\big) \le K.
\]
Since $y$ and $\pi_{\gamma}(x)$ both lie in a $K$-bi-quasigeodesic $\gamma$, Inequality \ref{eqn:biQuasiRev2} implies  $d\big(\pi_{\gamma}(x), y \big) \le 2K^{3} + K$.
\end{proof}

Let $t$ and $t'$ be the infimum and the supremum of $\eta^{-1}(S)$. Then $\eta|_{(-\infty, t]}$ is a geodesic disjoint from $N$ so $\diam\big(\pi_{\gamma}(\eta|_{(-\infty, t]})\big) < K$ holds. Furthermore, $\eta(t)$ belongs to $\eta \cap \partial N \subseteq S$, so there exists $y \in \gamma$ such that $d^{sym}(\eta(t), y) = K$. Then Observation \ref{obs:coarseEquivObs1} implies \[\begin{aligned}
d\big(\eta(t), \pi_{\gamma}(\eta(t))\big) &\le d(\eta(t), y) + d\big(y, \pi_{\gamma}(\eta(t)) \big) \le 2K,\\
d\big(\pi_{\gamma}(\eta(t)), \eta(t) \big) &\le d\big(\pi_{\gamma}(\eta(t)), y\big) + d(y, \eta(t)) \le 2K^{3} + 2K.
\end{aligned}
\]
We deduce that \[\begin{aligned}
\diam \big(\eta(t) \cup \pi_{\gamma}(\eta|_{(-\infty, t]}) \big) &\le d^{sym}\big(\eta(t), \pi_{\gamma}(\eta(t))\big) + \diam \big( \pi_{\gamma} (\eta|_{(-\infty, t]}) \big) \\
&< 2K^{3} + 4K + K \le 7K^{3}.
\end{aligned}
\]
For a similar reason, we have $\diam \big( \eta(t') \cup \pi_{\gamma}(\eta|_{[t', +\infty)})\big) < 7K^{3}$.

We next claim that $d^{sym}(\eta(s), \pi_{\gamma}(\eta)) \le 20K^{3}$ for each $s \in [t, t']$. If $\eta(s) \in S$, then there exists $y \in \gamma$ such that $d^{sym}(\eta(s), y) = K$. Then Observation  \ref{obs:coarseEquivObs1} implies that \begin{equation}\label{eqn:coarseEquivEasy}\begin{aligned}
d^{sym}\big(\eta(s), \pi_{\gamma}(\eta(s))\big) &\le d^{sym}(\eta(s), y) + d^{sym}\big(y, \pi_{\gamma}(\eta(s))\big)\\
& \le K + (K + 2K^{3} + K) \le 7K^{3}.
\end{aligned}
\end{equation}
If $\eta(s) \notin S$, then we take the connected component $(s_{1}, s_{2}) \in [t, t'] \setminus \eta^{-1}(S)$ of $s$. Then the geodesic $\eta|_{[s_{1}, s_{2}]}$ is disjoint from the $K$-neighborhood of $\gamma$ so the diameter of $\pi_{\gamma}(\eta|_{[s_{1}, s_{2}]})$ is at most $K$. Moreover, $\eta(s_{1})$ and $\eta(s_{2})$ belong to $\eta \cap \partial N$: let $p, q \in \gamma$ be points such that $d^{sym}(\eta(s_{1}), p), d^{sym}(\eta(s_{2}), q) = K$. By Observation \ref{obs:coarseEquivObs1}, we have\[
\begin{aligned}
d\big(\pi_{\gamma}(\eta(s_{1})), \eta(s)\big) &\le d\big(\pi_{\gamma}(\eta(s_{1})), \eta(s_{1})\big) + d\big(\eta(s_{1}), \eta(s) \big) \\
&\le d\big( \pi_{\gamma}(\eta(s_{1})), \eta(s_{1}) \big) + d\big(\eta(s_{1}), \eta(s_{2}) \big) \\
&\le d\big(\pi_{\gamma}(\eta(s_{1})), \eta(s_{1})\big) + d\big(\eta(s_{1}), \pi_{\gamma}(\eta(s_{1})) \big) \\
&\quad+ \diam \big(\pi_{\gamma}(\eta|_{[s_{1}, s_{2}]}) \big) + d\big(\pi_{\gamma}(\eta(s_{2})), \eta(s_{2}) \big)\\
&\le d\big(\pi_{\gamma}(\eta(s_{1})), p\big)  + d(p, \eta(s_{1})) + d(\eta(s_{1}), \gamma)\\
&\quad + \diam \big(\pi_{\gamma}(\eta|_{[s_{1}, s_{2}]}) \big)+  d\big(\pi_{\gamma}(\eta(s_{2})), q\big) + d\big(q, \eta(s_{2}) \big)\\
&\le(2K^{3} + K) + K + K + (2K^{3} + K) + K.
\end{aligned}
\]
Also, we have \[
\begin{aligned}
d\big(\eta(s), \pi_{\gamma}(\eta(s_{1})) \big) &\le d(\eta(s), \eta(s_{2})) + d\big(\eta(s_{2}), \pi_{\gamma}(\eta(s_{2})) \big) + d\big(\pi_{\gamma}(\eta(s_{2})), \pi_{\gamma}(\eta(s_{1})) \big) \\
&\le d(\eta(s_{1}), \eta(s_{2})) + d(\eta(s_{2}), \gamma\big) + \diam\big(\pi_{\gamma}(\eta|_{[s_{1}, s_{2}]}) \big)  \\
&\le \Big(d\big(\eta(s_{1}), \gamma \big)+ \diam \big(\pi_{\gamma}(\eta|_{[s_{1}, s_{2}]}) \big)  +  d\big(\pi_{\gamma}(\eta(s_{2})), q\big) + d\big(q, \eta(s_{2}) \big)\Big) \\
&+d\big(\eta(s_{2}), \gamma\big) + \diam\big(\pi_{\gamma}(\eta|_{[s_{1}, s_{2}]}) \big) \\
&\le K + K+ (2K^{3} + K) + K  +K.
\end{aligned}
\]
In summary, we have $d^{sym}(p, \eta(s)) \le 20K^{3}$ as desired.

Next, we claim that $d^{sym}\big(\pi_{\gamma}(\eta(s)), \eta|_{[t, t']}\big) < 20K^{3}$ for each $s \in I$. If $\eta(s) \in S$, then $d^{sym}\big(\pi_{\gamma}(\eta(s)), \eta(s)\big) \le 7K^{3}$ by Inequality \ref{eqn:coarseEquivEasy}.

If $\eta(s) \notin S$, then we take the connected component $(s_{1}, s_{2}) \in I \setminus \eta^{-1}(S)$ of $s$, with either $\eta(s_{1})$ or $\eta(s_{2})$ being in $\partial N$. Without loss of generality, suppose that $\eta(s_{2}) \in  \partial N$ and let $y \in \gamma$ be such that $d^{sym}(\eta(s_{2}), y) = K$. Observation \ref{obs:coarseEquivObs1} and the $K$-BGIP of $\gamma$ implies that \[\begin{aligned}
d^{sym}\big(\pi_{\gamma}(\eta(s)), \eta(s_{2}) \big) &\le 2 \diam \big(\pi_{\gamma}(\eta|_{[s, s_{2}]})\big) + d^{sym}\big(\pi_{\gamma}(\eta(s_{2})), y) + d^{sym}(y, \eta(s_{2}) \big) \\
&\le 2K + (2K^{3} + 2K) + K \le 10K^{3}.\qedhere
\end{aligned}
\]
\end{proof}

\begin{cor}\label{cor:coarseLipschitz}
Let $K>1$ and let $\gamma$ be a $K$-BGIP axis. Then the closest point projection $\pi_{\gamma}(\cdot) : X \rightarrow \gamma$ is $(1, 14K^{3})$-coarsely Lipschitz, i.e., \[
d(\pi_{\gamma}(x), \pi_{\gamma}(y) \big) \le d(x, y) + 14K^{3} \quad (\forall \, x, y \in X).
\]
\end{cor}

\begin{proof}
Let $\eta : I \rightarrow X$ represent the geodesic from $x$ to $y$. If $\pi_{\gamma}(\eta)$ has diameter smaller than $K$, the conclusion follows. If not, then by Lemma \ref{lem:coarseEquiv}, there exist $t < t'$ in $I$ such that $d(\pi_{\gamma}(x), \eta(t)) < 7K^{3}$ and $d(\eta(t'), \pi_{\gamma}(y)) < 7K^{3}$. We then have \[\begin{aligned}
d(\pi_{\gamma}(x), \pi_{\gamma}(y)) &\le d(\pi_{\gamma}(x), \eta(t)) + d(\eta(t), \eta(t')) + d(\eta(t'), \pi_{\gamma}(y)) \\
&< 7K^{3} + \diam(I) + 7K^{3} \le d(x, y) + 14K^{3}.\qedhere
\end{aligned}
\]
\end{proof}

The following lemma was proved in \cite[Section 3.1]{choi2022random1} for symmetric metrics; the same proof works here as well, after replacing quasi-geodesics with bi-quasigeodesics.

\begin{lem}[{\cite[Lemma 3.2]{choi2022random1}}] \label{lem:coarseGeod}
For each $K>1$ there exists a constant $K'>K$ such that the following holds.

Let $\eta : J \rightarrow X$ be a $K$-bi-quasigeodesic whose endpoints are $x$ and $y$, let $A \subseteq \eta$ be a subset of $\eta$ such that $d(x, A) < K$, $d(y, A) < K$, and let $\gamma : J' \rightarrow X$ be a geodesic that is $K$-coarsely equivalent to $A$. Then $\eta$ and $\gamma$ are also $K'$-coarsely equivalent, and moreover, there exists a $K'$-quasi-isometry $\varphi : J \rightarrow J'$ such that $d^{sym}(\eta(t), (\gamma \circ \varphi)(t)) < K'$ for each $t \in J$.
\end{lem}

\begin{cor}[{\cite[Corollary 3.3]{choi2022random1}}] \label{cor:coarseGeod}
For each $K>1$ there exists a constant $K'>K$ such that the following holds.

Let $\eta : J \rightarrow X$ be a $K$-BGIP axis and let $\gamma : J' \rightarrow X$ be a geodesic connecting the endpoints of $\eta$. Then there exists a $K'$-quasi-isometry $\varphi : J \rightarrow J'$ such that $d^{sym}(\eta(t), (\gamma \circ \varphi)(t)) < K'$ for each $t \in J$; in particular, $\eta$ and $\gamma$ are $K'$-coarsely equivalent.
\end{cor}

By using Corollary \ref{cor:coarseGeod} and Corollary \ref{cor:coarseLipschitz}, we deduced the following in \cite[Section 3.1]{choi2022random1}; the proof there is phrased in terms of symmetric metrics, but the same proof works for asymmetric metrics as well.

\begin{cor}[{\cite[Corollary 3.4]{choi2022random1}}] \label{cor:largeProj}
For each $K>1$ there exists a constant $K' = K'(K)$ that satisfies the following.

Let $\kappa : I \rightarrow X$ and $\eta : J \rightarrow X$ be $K$-BGIP axes. Suppose that $\diam(\pi_{\kappa}(\eta)) > K'$. Then there exist $t < t'$ in $I$ and $s < s'$ in $J$ such that the following sets are all $K'$-coarsely equivalent: \[
\kappa|_{[t, t']}, \,\, \eta|_{[s, s']},\,\, \pi_{\kappa}(\eta), \,\, \pi_{\eta}(\kappa).
\]
Moreover, we have \[
\diam\big(\pi_{\kappa}(\eta|_{(-\infty, s]}) \cup \eta(s) \big) < K', \quad \diam\big(\pi_{\kappa}(\eta|_{[s', +\infty)} ) \cup \eta(s') \big) < K'.
\]
\end{cor}

The following fact is proven in \cite[Proposition 2.2.(3)]{yang2020genericity}. For completeness, we sketch the proof in the setting of asymmetric metric spaces.

\begin{lem}\label{lem:BGIPRestriction}
For each $K>1$, there exists $K'= K'(K)>1$ such that every subpath of a $K$-BGIP axis is again a $K'$-BGIP axis.
\end{lem}

\begin{proof}
Let $\gamma: I \rightarrow X$ be a $K$-BGIP axis and let $\kappa = \gamma|_{[a, b]} : [a, b] \rightarrow X$ be the restriction for some $a< b$ in $I$. Our goal is to find a constant $K'$, depending only on $K$, such that $\diam(\pi_{\kappa}(\eta)) > K'$  implies $\eta \cap N_{K'}(\kappa) \neq \emptyset$ for each geodesic $\eta = [x, y]$. Without loss of generality, we can assume that $\kappa$ is longer than a large enough constant.

Before the proof, we make a simple observation: for $x \in X$, if $\pi_{\gamma}(x)$ intersects $\kappa$, then $\pi_{\kappa}(x)$ equals $\pi_{\gamma}(x) \cap \kappa$, which is nonempty.

Let $x, y \in X$. First consider the case that $\pi_{\gamma}(x)$ is far from $\gamma \setminus \kappa$ and deep inside $\kappa$. In this case, Lemma \ref{lem:coarseEquiv} tells us that $[x, y]$ passes through a bounded neighborhood of $\pi_{\gamma}(x)$ unless $\pi_{\gamma}(y)$ is $K$-close to $\pi_{\gamma}(x)$, in which case $\pi_{\gamma}(x)$ and $\pi_{\gamma}(y)$ are both contained in $\kappa$ and $\diam(\pi_{\kappa}(x) \cup \pi_{\kappa}(y)) = \diam(\pi_{\gamma}(x) \cup \pi_{\gamma}(y)) \le K$.  Similarly, when $\pi_{\gamma}(y)$ lies deep in $\kappa$, then $\diam(\pi_{\kappa}(x) \cup \pi_{\kappa}(y)) > K$ implies that $[x, y]$ passes nearby $\kappa$.

Now consider the case that $\pi_{\gamma}(x)$ and $\pi_{\gamma}(y)$ are both near $\gamma \setminus \kappa$. Let $\gamma_{L}$ and $\gamma_{R}$ be the left and the right components of $\gamma \setminus \kappa$, respectively. If $\pi_{\gamma}(x)$ $(\pi_{\gamma}(y)$, resp.) intersects with a bounded neighborhood of $\gamma_{L}$ ($\gamma_{R}$, resp.), then they are far away (because $\kappa$ is assumed to be long enough). Lemma \ref{lem:coarseEquiv} says that some subgeodesic $\eta'$ of $[x, y]$ is coarsely equivalent to a subset of $\gamma$ that coarsely connects $\gamma_{L}$ and $\gamma_{R}$. Lemma \ref{lem:coarseGeod} then tells us that $\eta'$ coarsely contains the entire subsegement of $\gamma$ in between $\gamma_{L}$ and $\gamma_{R}$, which is $\kappa$. Hence, $[x, y]$ passes nearby $\kappa$.

By symmetry, it remains to handle the case that $\pi_{\gamma}(x)$ and $\pi_{\gamma}(y)$ are both contained in a bounded neighborhood $\gamma_{L}$. In this case, we claim that $\pi_{\kappa}(x)$ is close to $\gamma(a)$. To see this, pick any $s \in [a, b]$ with large $s-a$. Also, pick $t \in (-\infty, a]$ such that $\gamma(t) \in \pi_{\gamma}(x)$. We now apply Lemma \ref{lem:coarseEquiv} to the geodesic $[x,  \gamma(s)]$ and obtain its subsegment $\eta'$ which is coarsely equivalent to a subset of $\gamma$ containing $\gamma(t)$ and $\gamma(s)$. Lemma \ref{lem:coarseGeod} then tells us that $\eta'$ is close to each $\gamma(u)$ for $t \le u \le s$, including $u = a$. Hence, $[x, \gamma(s)]$ passes nearby $\gamma(a)$, with $\gamma(s)$ and $\gamma(a)$ being far away from each other (since $s-a$ is large). This implies that $\gamma(a)$ is closer than $\gamma(s)$ to $x$, and $\gamma(s) \notin \pi_{\kappa}(x)$. This implies that $\pi_{\kappa}(x)$ consists of points close to $\gamma(a)$ as desired. 

For the same reason, $\pi_{\kappa}(y)$ is also close to $\gamma(a)$ and $\pi_{\kappa}(x) \cup \pi_{\kappa}(y)$ has small diameter. This ends the proof.
\end{proof}

\subsection{BGIP axes and alignment} \label{subsection:BGIPAlign}

We are now ready to discuss the alignment of BGIP axes.

\begin{definition}[{\cite[Definition 3.5]{choi2022random1}}]\label{dfn:alignment}
Given paths $\kappa_{i}$ from $x_{i}$ to $x'_{i}$ for each $i=1, \ldots, n$, we say that $(\kappa_{1}, \ldots, \kappa_{n})$ is $C$-aligned if \[
\diam\left(x_{i}' \cup \pi_{\kappa_{i}}(\kappa_{i+1})\right) < C,\quad \diam\left(x_{i+1} \cup \pi_{\kappa_{i+1}}(\kappa_{i})\right) < C.
\]
hold for $i=1, \ldots, n-1$.
\end{definition}

Here, we regard points as degenerate paths. For example, for a point $x$ and a path $\gamma$, we say that $(x, \gamma)$ is $C$-aligned if \[
\diam \big(\textrm{beginning point of $\gamma$} \cup \pi_{\gamma}(x)\big) < C.
\]

\begin{obs}[{\cite[Observation 2.4]{choi2022random2}}]\label{obs:BGIPAlignRev}
Let $g$ be an isometry of $X$, let $1 \le k \le n$, let $K>0$ and let $\gamma_{1}, \ldots, \gamma_{n}$ be paths on $X$.
\begin{enumerate}
\item If $(\gamma_{1}, \ldots, \gamma_{k})$ and $(\gamma_{k}, \ldots, \gamma_{n})$ are $K$-aligned, then there concatenation $(\gamma_{1}, \ldots, \gamma_{n})$ is also $K$-aligned.
\item If $(\gamma_{1}, \ldots, \gamma_{n})$ is $K$-aligned, then $(\bar{\gamma}_{n}, \ldots, \bar{\gamma}_{1})$ is also $K$-aligned.
\item If $(\gamma_{1}, \ldots, \gamma_{n})$ is $K$-aligned, then $(g\gamma_{1}, \ldots, g\gamma_{n})$ is also $K$-aligned.
\end{enumerate}
\end{obs}

By combining Lemma \ref{lem:coarseEquiv} and Lemma \ref{lem:coarseGeod}, we deduce: 

\begin{cor}[{\cite[Corollary 3.6]{choi2022random1}}]\label{cor:1segmentAlign}
For each $C, K>1$, there exists $K' = K'(K, C) > K, C$ that satisfies the following.

Let $x, y \in X$ and let $\kappa$ be a $K$-BGIP axis such that $\diam(\kappa) > K'$ and such that $(x, \kappa, y)$ is $C$-aligned. Then $[x, y]$ contains a subsegment $\eta$ that is contained in the $20K^{3}$-$d^{sym}$-neighborhood of $\kappa$ and is $K'$-fellow traveling with $\kappa$.
\end{cor}

The following was proved in \cite{choi2022random1} for symmetric metrics. Using Lemma \ref{lem:coarseEquiv} and Corollary \ref{cor:coarseLipschitz} this time, the same proof works for asymmetric metrics as well.

\begin{lem}[{\cite[Lemma 3.7]{choi2022random1}}]\label{lem:1segment}
For each $C>0$ and $K>1$, there exists $D= D(K, C)>C$ that satisfies the following property.

Let $\kappa, \eta$ be $K$-BGIP axes whose beginning points are $x$ and $x'$, respectively. Suppose that $(\kappa, x')$ and $(x, \eta)$ are $C$-aligned. Then $(\kappa, \eta)$ is $D$-aligned.
\end{lem}

In \cite{sisto2018contracting}, Sisto proved the following property for constricting geodesics with respect to special paths. The same proof works for BGIP axes; we give a rephrasing of Sisto's proof for convenience.

\begin{lem}[{\cite[Lemma 2.5]{sisto2018contracting}}] \label{lem:Behrstock}
For each $D>0$ and $K>1$, there exists $E = E(K, D) > K, D$ that satisfies the following.

Let $\kappa$ and $\eta$ be $K$-BGIP axes in $X$. Suppose that $(\kappa, \eta)$ is $D$-aligned. Then for any $p \in X$, either $(p, \eta)$ is $E$-aligned or $(\kappa, p)$ is $E$-aligned.
\end{lem}

\begin{proof}
Let $I = [a, b]$ and $J = [c, d]$ be the domains of $\kappa$ and $\eta$, respectively. The assumption says that $\pi_{\kappa}(\eta)$ is close to $\kappa(b)$ and $\pi_{\eta}(\kappa)$ is close to $\eta(c)$. Now, consider the geodesic $\gamma = [p, \kappa(b)]$ from $p$ to $\kappa(b)$, and pick the earliest point $q \in \gamma$ that belongs to the $K$-closed neighborhood of $\kappa \cup \eta$. (Such point exists because $\kappa(b) \in N_{K}(\kappa \cup \eta)$.)\begin{enumerate}
\item If $q \in \overline{\mathcal{N}_{K}(\kappa)}$, then $d^{sym}(q, \kappa) = K$ and $[x, q]$ does not intersect the $K$-neighborhood of $\eta$. Then $\pi_{\eta}(x)$ is contained in the $\diam(\pi_{\eta}([x, q]))$-neighborhood of $\pi_{\eta}(q)$, and $\pi_{\eta}([x, q])$ has diameter at most $K$ by the $K$-BGIP of $\kappa$. Now, $\pi_{\eta}(q)$ is $\big(d(q, \kappa) + 14K^{3}\big)$-close to $\pi_{\eta}(\kappa)$ by Corollary \ref{cor:coarseLipschitz}, which is close to $\eta(c)$. It follows that $\pi_{\eta}(x)$ is contained in a uniform neighborhood of $\eta(c)$, and $(x, \eta)$ is aligned.
\item If $q \in \overline{\mathcal{N}_{K}(\eta)}$, then $d^{sym}(q, \eta) = K$ and $[x, q]$ is does not intersect the $K$-neighborhood of $\kappa$. By a symmetric argument, $\pi_{\kappa}(x)$ is uniformly close to $\kappa(b)$ and $(\kappa, x)$ is aligned.\qedhere
\end{enumerate}
\end{proof}

Using Corollary \ref{cor:1segmentAlign} and Lemma \ref{lem:Behrstock} as ingredients, we proved the following in \cite{choi2022random1}:

\begin{prop}[{\cite[Proposition 3.11]{choi2022random1}}]\label{prop:BGIPWitness}
For each $D>0$ and $K>1$, there exist $E = E(K, D)>K, D$ and $L = L(K,D)>K, D$ that satisfy the following. 

Let $x$ and $y$ be points in $X$ and let $\kappa_{1}, \ldots, \kappa_{n}$ be $K$-BGIP axes whose domains are longer than $L$ and such that $(x, \kappa_{1}, \ldots, \kappa_{n}, y)$ is $D$-aligned. Then the geodesic $[x, y]$ has subsegments $\eta_{1}, \ldots, \eta_{n}$, in order from left to right, that are longer than $100E$ and such that $\eta_{i}$ and $\kappa_{i}$ are $0.1E$-fellow traveling for each $i$.  In particular, $(x, \kappa_{i}, y)$ are $E$-aligned for each $i$.
\end{prop}

Often, an isometry $g \in \Isom(X)$ gives rise to a \emph{periodic} aligned sequence of BGIP axes $( \ldots, g^{i-1} \gamma_{n}, g^{i} \gamma_{1}, \ldots, g^{i} \gamma_{n}, g^{i+1} \gamma_{1}, \ldots)$. The following lemma enables to conclude that $g$ is an BGIP isometry in this situation.

\begin{lem}[{\cite[Proposition 2.9]{yang2019statistically}}]\label{lem:BGIPConcat}
For each $D, M>0$ and $K>1$, there exist $E = E(K, D, M) > D$ and $L = L(K, D) > D$ that satisfies the following.

Let $\kappa_{1}, \ldots, \kappa_{n}$ be $K$-BGIP axes whose domains are longer than $L$. Suppose that $(\kappa_{1}, \ldots, \kappa_{n})$ is $D$-aligned and $d(\kappa_{i}, \kappa_{i+1}) < M$ for each $i$. Then the concatenation $\kappa_{1} \cup \ldots \cup \kappa_{n}$ of $\kappa_{1}, \ldots, \kappa_{n}$ is an $E$-BGIP axis.
\end{lem}

\begin{lem}[{\cite[Proposition 2.9]{yang2019statistically}}]\label{lem:BGIPConcat2}
For each $D, M>0$ and $K>1$, there exist $E = E(K, D, M) > D$ and $L = L(K, D) > D$ that satisfies the following.

Let $\kappa_{1}, \ldots, \kappa_{n}$ be $K$-BGIP axes whose domains are longer than $L$, and suppose that $\kappa_{i}$ is connecting $y_{i-1}$ to $x_{i}$. Suppose that $(x_{0}, \kappa_{1}, \ldots, \kappa_{n})$ is $D$-aligned. Then the concatenation $[x_{0}, y_{0}] \cup \kappa_{1} \cup [x_{1}, y_{1}] \cup \kappa_{2} \cup \ldots$ is an $E$-quasigeodesic.
\end{lem}
Here, note that the resulting path in Lemma \ref{lem:BGIPConcat2} may not be bi-quasigeodesic.

We now recall the concept of Schottky sets. Given a sequence $\alpha  = (a_{1}, \ldots, a_{n}) \in G^{n}$, we employ the following notations: \[\begin{aligned}
\Pi(\alpha) &:= a_{1} a_{2} \cdots a_{n}, \\
\Gamma(\alpha) &:= \big(o, a_{1} o, a_{1} a_{2} o, \ldots, \Pi(\alpha) o\big).
\end{aligned}
\]

\begin{definition}[{\cite[Definition 3.14]{choi2022random1}}]\label{dfn:Schottky}
Let $K > 0$ and $S \subseteq G^{M}$ be a set of sequences of $M$ isometries. We say that $S$ is \emph{$K$-Schottky} if the following hold: \begin{enumerate}
\item $\Gamma(\alpha)$ is a $K$-BGIP axis for all $\alpha \in S$;
\item for each $x \in X$, we have \[
\#\Big\{ \alpha \in S : \textrm{$\left(x, \Gamma(\alpha)\right)$ and $\left(\Gamma(\alpha), \Pi(\alpha)x \right)$ are $K$-aligned} \Big\} \ge \#S - 1;
\]
\item for each $\alpha\in S$, $\left(\Gamma(\alpha), \Pi(\alpha)\Gamma(\alpha) \right)$ is $K$-aligned.
\end{enumerate}
We say that $S$ is \emph{large enough} if its cardinality is at least 400.
\end{definition}

\begin{definition}[{\cite[Definition 2.8]{choi2022random2}}]  \label{dfn:SchottkyLong}
Given a constant $K_{0}>0$, we define: \begin{itemize}
\item $D_{0}= D(K_{0}, K_{0})$ be as in Lemma \ref{lem:1segment}, 
\item $D_{1} = E(K_{0}, D_{0})$, $L = L(K_{0}, D_{0})$ be as in Proposition \ref{prop:BGIPWitness},
\item $E_{0}= E(K_{0}, D_{1})$, $L' = L(K_{0}, D_{1})$ be as in Proposition \ref{prop:BGIPWitness},
\item $L' = L(K_{0}, D_{1})$ be as in Lemma \ref{lem:BGIPConcat}.
\end{itemize}
A $K_{0}$-Schottky set $S \subseteq G^{n}$ is called a \emph{fairly long $K_{0}$-Schottky set} if: \begin{enumerate}
\item $n > \max \{L, L', L''\}$, and
\item $d(o, \prodSeq(\alpha) o) \ge 10E_{0}$ for all $\alpha \in S$.
\end{enumerate}
When the Schottky parameter $K_{0}$ is understood, the constants $D_{0}, D_{1}, E_{0}$ always denote the ones defined above. Once a fairly long Schottky set $S$ is understood, its element $s$ is called a \emph{Schottky sequence} and the translates of $\Gamma^{\pm}(s)$ are called \emph{Schottky axes}. When a probability measure $\mu$ on $G$ is given in addition such that $S \subseteq (\supp \mu)^{n}$, we say that $S$ is a fairly long Schottky set for $\mu$.
\end{definition}

\begin{definition}[{\cite[Definition 3.16]{choi2022random1}}]  \label{dfn:semiAlign}
Let $S$ be a fairly long Schottky set and let $K>0$. We say that a sequence of Schottky axes is $K$-semi-aligned if it is a subsequence of a $K$-aligned sequence of Schottky axes.

More precisely, for Schottky axes $\gamma_{1}, \ldots, \gamma_{n}$, we say that $(\gamma_{1}, \ldots, \gamma_{n})$ is \emph{$K$-semi-aligned} if there exist Schottky axes $\eta_{1}, \ldots, \eta_{m}$ such that $(\eta_{1}, \ldots, \eta_{m})$ is $K$-aligned and if there exists a subsequence $\{i(1) < \ldots < i(n)\} \subseteq \{1, \ldots, m\}$ such that $\gamma_{l} = \eta_{i(l)}$ for $l=1, \ldots, n$.

Similarly, for points $x, y \in X$ and Schottky axes $\gamma_{1}, \ldots, \gamma_{n}$, we say that $(x, \gamma_{1}, \ldots, \gamma_{n}, y)$ is \emph{$K$-semi-aligned} if it is a subsequence of a $K$-aligned sequence $(x, \eta_{1}, \ldots, \eta_{m}, y)$ for some Schottky axes $\eta_{1}, \ldots, \eta_{m}$.
\end{definition}

Proposition \ref{prop:BGIPWitness} implies the following corollary.

\begin{cor}\label{cor:semiAlign}
Let $S$ be a fairly long $K_{0}$-Schottky set, with constants $D_{0}, D_{1}, E_{0}$ as in Definition \ref{dfn:Schottky}. Let $x, y \in X$ and $\gamma_{1}, \ldots, \gamma_{N}$ be Schottky axes. Then $(x, \gamma_{1}, \ldots, \gamma_{N}, y)$  is $D_{1}$-aligned $(E_{0}$, resp.) whenever it is $D_{0}$-semi-aligned ($D_{1}$-semi-aligned, resp.), and moreover, $[x, y]$ contains subsegments $\eta_{1}, \ldots, \eta_{N}$, each longer than $100D_{1}$ ($100E_{0}$, resp.) and in order from closest to farthest from $x$, such that $\eta_{i}$ and $\gamma_{i}$ are $0.1D_{1}$-fellow traveling ($0.1E_{0}$-fellow traveling, resp.) for each $i$.
\end{cor}

Using Lemma \ref{lem:BGIPRestriction}, Lemma \ref{lem:1segment} and Proposition \ref{prop:BGIPWitness}, we proved the following in \cite{choi2022random1}:

\begin{prop}[{\cite[Proposition 3.12]{choi2022random1}}]\label{prop:Schottky} Let $(X, G, o)$ be as in Convention \ref{conv:main} and let $\mu$ be a non-elementary probability measure on $G$. Then for each $N, L > 0$, there exists $K=K(N) > 0$ and $n>L$ such that there exists a $K$-Schottky set of cardinality $N$ in $(\supp \mu)^{n}$. 
\end{prop}

\section{Limit laws} \label{section:limit}

Recall again that we are fixing an asymmetric metric space $X$ and its countable isometry group $G$ involving two independent BGIP isometries. We will now describe the results established in Section 4, 5 and 6 of \cite{choi2022random1}.

\subsection{Results from \cite{choi2022random1}} \label{subsection:limit1}

In \cite[Section 5]{choi2022random1}, the author constructed pivotal times, first in a discrete model and next in random walks, following the idea of Gou{\"e}zel \cite[Section 4A]{gouezel2022exponential}. The arguments in \cite[Section 5]{choi2022random1} relied the following ingredients: \begin{enumerate}
\item Existence of large and fairly long Schottky set for non-elementary probability measure $\mu$ (Proposition \ref{prop:Schottky}),
\item Alignments of points and Schottky axes (the properties of Schottky set described in Definition \ref{dfn:Schottky}), and
\item Alignment lemma that guarantees the alignment of two Schottky axes given the alignments of one Schottky axis and an endpoint of another Schottky axis (Lemma \ref{lem:1segment}).
\end{enumerate}
Since we have the versions of alignment and Schottky set for BGIP isometries, the entire proof works verbatim. As a consequence, we obtain the following proposition. Recall again the notation \[
\mathbf{Y}_{i}(\w) := \big(Z_{i-M_{0}}(\w) o, Z_{i - M_{0} + 1}(\w) o, \ldots, Z_{i-1}(\w) o, Z_{i} (\w) o \big).
\]

\begin{definition}[{\cite[Definition 4.1]{choi2022random1}}]\label{dfn:pivotClass}
Let $(X, G, o)$ be as in Convention \ref{conv:main}, let $\mu$ be a non-elementary probability measure on $G$, let $(\Omega, \Prob)$ be a probability space for $\mu$, let $K_{0}, M_{0} > 0$ and let $S$ be a fairly long $K_{0}$-Schottky set contained in $(\supp \mu)^{M_{0}}$. 

A subset $\mathcal{E}$ of $\Omega$ is called a pivotal equivalence class if there exists a set $\mathcal{P}(\mathcal{E}) = \{j(1) < j(2) < \ldots \} \subseteq M_{0} \Z_{>0}$, called the \emph{set of pivotal times}, such that the following hold: \begin{enumerate}
\item for each $i \notin \{j(k) - l : k \ge 1, l = 0, \ldots, M_{0} - 1 \}$, $g_{i}(\w)$ is fixed on $\mathcal{E}$;
\item for each $\w \in \mathcal{E}$ and $k \ge 1$, $s_{k}(\w) := \big( g_{j(k) - M_{0} + 1}(\w), g_{j(k) - M_{0} + 2} (\w), \ldots, g_{j(k)} (\w) \big)$ is a Schottky sequence;
\item for each $\w \in \mathcal{E}$, the sequence $\big(o, \mathbf{Y}_{j(1)} (\w) , \mathbf{Y}_{j(2)} (\w), \ldots \big)$ is $D_{0}$-semi-aligned;
\item on $\mathcal{E}$, $\{s_{1}(\w), s_{2}(\w), \ldots\}$ are i.i.d.s distributed according to the uniform measure on $S$.
\end{enumerate}
For $\w \in \mathcal{E}$, we call $\mathcal{P}(\mathcal{E})$ the \emph{set of pivotal times for $\w$} and write it as $\mathcal{P}(\w)$.
\end{definition}

We say that a pivotal equivalence class $\mathcal{E}$ \emph{avoids} an integer $k$ if $k$ is not in $\{ j - l: j \in \mathcal{P}(\mathcal{E}), l=0, \ldots, M_{0} - 1\}$.

\begin{prop}[{\cite[Proposition 4.2]{choi2022random1}}] \label{prop:pivotClassProp}
Let $(X, G, o)$ be as in Convention \ref{conv:main}, let $\mu$ be a non-elementary probabiltiy measure on $G$ and let $S$ be a large and fairly long Schottky set for $\mu$. Then there exist a probability space $(\Omega, \Prob)$ for $\mu$ and a constant $K>0$ such that, for each $n \ge 0$, we have a measurable partition $\mathscr{P}_{n} = \{\mathcal{E}_{\alpha}\}_{\alpha}$ of $\Omega$ by pivotal equivalence classes avoiding $1$, $\ldots$, $\lfloor n/2 \rfloor+ 1$ and $n+1$, that satisfies \[
\Prob \left( \w : \# (\mathcal{P}(\w) \cap \{1, \ldots, k\} ) \le k / K \, | \, g_{1}, \ldots, g_{\lfloor n/2  \rfloor + 1}, g_{n+1} \right) \le K e^{-k/K}
\]
for each choice of $g_{1}, \ldots, g_{\lfloor n/2 \rfloor + 1}, g_{n+1} \in G$ and for each $k \ge n$.
\end{prop}

Furthermore, again by using the property of Schottky set (Definition \ref{dfn:Schottky}) and the alignment lemma (Lemma \ref{lem:1segment}), we obtained the following:

\begin{cor}[{\cite[Corollary 4.5]{choi2022random1}}]\label{cor:pivotingRW1}
Let $(X, G, o)$ be as in Convention \ref{conv:main}, let $\mu$ be a non-elementary probability measure on $G$, let $K_{0}, N_{0}>0$ and let  $S$ be a long enough $K_{0}$-Schottky set for $\mu$ with cardinality $N_{0}$. Let $\mathcal{E}$ be a pivotal equivalence class for $\mu$ with $\diffPivot(\mathcal{E}) = \{j(1) < j(2) < \ldots\}$ and let $x$ be a point in $X$. Then for each $k \ge 1$ we have \[
\Prob\left( \big(x, \,\axes_{j(k)}(\w),\, \axes_{j(k+1)}(\w), \ldots \big) \,\, \textrm{is $D_{0}$-semi-aligned} \, \Big| \, \mathcal{E} \right) \ge 1- (1/N_{0})^{k}.
\]
Moreover, for any $m\ge 1$, $n \ge j(m)$ and $k = 1, \ldots, m$, we have \[
\Prob\left( \big(\axes_{j(1)}(\w), \,\ldots, \,\axes_{j(m - k+1)}(\w),\, Z_{n}(\w) o \big) \,\, \textrm{is $D_{0}$-semi-aligned} \, \Big| \, \mathcal{E} \right) \ge 1- (1/N_{0})^{k}.
\]
\end{cor}

\begin{cor}[{\cite[Corollary 4.6]{choi2022random1}}]\label{cor:pivotingRW2}
Let $(X, G, o)$ be as in Convention \ref{conv:main}, let $\mu$ be a non-elementary probability measure on $G$ and $\check{\mu}$ be its reflected version, let $K_{0}, N_{0}>0$ and let $S$ and $\check{S}$ be long enough $K_{0}$-Schottky sets for $\mu$ and $\check{\mu}$, respectively, with cardinality $N_{0}$. Let $\mathcal{E}$ be a pivotal equivalence class for $\mu$ with $\diffPivot(\mathcal{E}) = \{j(1) < j(2) < \ldots\}$, and let $\check{\mathcal{E}}$ be a pivotal equivalence class for $\check{\mu}$ with $\diffPivot(\check{\mathcal{E}}) = \{\check{j}(1) < \check{j}(2) < \ldots\}$.  Then for each $k \ge 1$ we have \[
\Prob\left( \big(\bar{\axes}_{j(k)}(\check{\w}),\, \axes_{\check{j}(k)}(\w)\big) \,\, \textrm{is $D_{0}$-semi-aligned} \, \Big| \, \mathcal{E} \right) \ge 1- (2/N_{0})^{k}.
\]
\end{cor}

Combining these, we proved in the proof of \cite[Corollary 4.7]{choi2022random1} that: 

\begin{prop}[{\cite[Corollary 4.7]{choi2022random1}}]\label{prop:pivotingRWSLLN1}
Let $(X, G, o)$ be as in Convention \ref{conv:main}, let $\mu$ be a non-elementary probability measure on $G$, let $(\Omega, \Prob)$ be a probability space for $\mu$, and let $S$ be a large and fairly long Schottky set for $\mu$. Then there exists $K>0$ such that \[
\Prob\left( \w : \begin{array}{c}\textrm{$\big(o, \mathbf{Y}_{j(1)}, \ldots, \mathbf{Y}_{j(\lfloor n/2K \rfloor)}, Z_{n}(\w) o\big)$ is $D_{0}$-semi-aligned} \\\textrm{ for some $j(1), \ldots, j(\lfloor n/ 2K \rfloor) \in \Z$} \end{array} \right) \ge 1 - K e^{-n/K}.
\]
\end{prop}

Now, Corollary \ref{cor:semiAlign} and our choices of the constants $D_{0}, D_{1}$ tells us that, if $(x, \kappa_{1}, \ldots, \kappa_{m}, y)$ is a $D_{0}$-semi-aligned sequence, then $d(x, y)$ is larger than $50 D_{1}m$. Combining this with Proposition \ref{prop:pivotingRWSLLN1}, we deduce the strict positivity of the escape rate: 

\begin{thm}[{\cite[Corollary 4.7]{choi2022random1}}]\label{thm:pivotingRWSLLN2}
Let $(X, G, o)$ be as in Convention \ref{conv:main} and let $(Z_{n})_{n}$ be the random walk generated by a non-elementary probability measure $\mu$ on $G$. Then there exists a strictly positive quantity $\lambda(\mu) \in (0, +\infty]$, called the \emph{drift} of $\mu$, such that  \[
\lambda(\mu) := \lim_{n \rightarrow \infty} \frac{1}{n} d(o, Z_{n} o) \quad \textrm{almost surely.}
\]
\end{thm}

We next turn to deviation inequalities. We use the parametrization $(\check{\w}, \check{w}) \in  (G^{\Z_{>0}}, \check{\mu}^{\Z_{>0}}) \times (G^{\Z_{>0}}, \mu^{\Z_{>0}})$. In \cite{choi2022random1}, we defined: 

\begin{definition}[{\cite[Section 4.3]{choi2022random1}}]\label{dfn:Devi}
Let $(X, G, o)$ be as in Convention \ref{conv:main}, let $\mu$ be a non-elementary probability measure on $G$ and let $S$ be a large and fairly long $K_{0}$-Schottky set in $(\supp \mu)^{M_{0}}$ for some $K_{0}, M_{0} > 0$. 

For $(\check{\w}, \w) \in  (G^{\Z_{>0}}, \check{\mu}^{\Z_{>0}}) \times (G^{\Z_{>0}}, \mu^{\Z_{>0}})$, we define $\Devi = \Devi(\check{\w}, \w)$ to be the minimal index such that there exists $M_{0} \le i \le \Devi$ satisfying: \begin{enumerate}
\item $\mathbf{Y}_{i}(\w) := (Z_{i-M_{0}}o, \ldots, Z_{i-1} o, Z_{i} o)$ is a Schottky axis;
\item $(\check{Z}_{m} o, \mathbf{Y}_{i}(\w), Z_{n} o)$ is $D_{0}$-semi-aligned for all $n \ge \Devi$ and $m \ge 0$.
\end{enumerate}

We then define $\check{\Devi} = \check{\Devi}(\check{\w}, \w)$ to be the minimal index such that there exists $M_{0} \le i \le \check{\Devi}$ satisfying: \begin{enumerate}
\item $\bar{\mathbf{Y}}_{i}(\check{\w}) := (\check{Z}_{i} o,\check{Z}_{i-1} o, \ldots, \check{Z}_{i-M_{0}} o)$ is a Schottky axis;
\item $(\check{Z}_{n} o, \bar{\mathbf{Y}}_{i}(\check{\w}), Z_{m} o)$ is $D_{0}$-semi-aligned for all $n \ge \check{\Devi}$ and $m \ge 0$.
\end{enumerate}
\end{definition}

Using Corollary \ref{cor:pivotingRW1} and \ref{cor:pivotingRW2}, we deduced the following: 

\begin{lem}[{\cite[Lemma 4.9]{choi2022random1}}]\label{lem:Devi}
Let $(X, G, o)$ be as in Convention \ref{conv:main}, let $\mu$ be a non-elementary probability measure on $G$ and let $S$ be a large and long enough Schottky set for $\mu$. Then there exists $K'>0$ such that 
\begin{equation}\label{eqn:DeviDecays1}
\Prob\left(\Devi(\check{\w}, \w) \ge k \, \Big|\, g_{k+1}, \check{g}_{1}, \ldots, \check{g}_{k+1}\right) \le K' e^{-k/K'}
\end{equation}
holds for every $k \ge 0$ and every choice of $g_{k+1}, \check{g}_{1}, \ldots, \check{g}_{k+1} \in G$, and \begin{equation}\label{eqn:DeviDecays2}
\Prob\left(\check{\Devi}(\check{\w}, \w) \ge k \, \Big|\, \check{g}_{k+1},g_{1}, \ldots, g_{k+1}\right) \le K' e^{-k/K'}
\end{equation}
holds for every $k \ge 0$ and every choice of $\check{g}_{k+1}, g_{1}, \ldots, g_{k+1} \in G$.
\end{lem}

This martingale-like estimates enable us to compute the moments of $d(o, Z_{\Devi}o)$ and $d(o, \check{Z}_{\check{\Devi}} o)$ based on the moments of $\mu$ and $\check{\mu}$, respectively. In \cite[Lemma 4.8]{choi2022random1}, these quantities were used to estimate the Gromov product between the backward and the forward sample paths. This lemma needs a modification in our setting due to the asymmetry of the metric: 

\begin{lem}[cf. {\cite[Lemma 4.8]{choi2022random1}}] \label{lem:GromEstimDevi}
Let $(X, G, o)$ be as in Convention \ref{conv:main}, let $\mu$ be a non-elementary probability measure on $G$ and let $S$ be a fairly long Schottky set for $\mu$. Then for each $(\check{\w}, \w) \in (G^{\Z_{>0}}, \check{\mu}^{\Z_{>0}}) \times (G^{\Z_{>0}}, \mu^{\Z_{>0}})$, we have \begin{equation}\label{eqn:GromEstimDeviMod1}
(\check{Z}_{m} o, Z_{n} o)_{o} \le \frac{1}{2} \left[d(Z_{k} o, o) + d(o, Z_{k} o) \right] = \frac{1}{2} d^{sym}(o, Z_{k} o)
\end{equation}
for all $m \ge 0$ and $n, k \ge \Devi(\check{\w}, \w)$. Moreover, we have \begin{equation}\label{eqn:GromEstimDeviMod2}
(\check{Z}_{m} o, Z_{n} o)_{o} \le \frac{1}{2} \left[d(\check{Z}_{\min (m, \check{\Devi})} o, o) + d(o ,Z_{\min (n, \Devi)} o) \right]
\end{equation}
for all $m, n \ge 0$.
\end{lem}

\begin{proof}
For the first assertion, let $i \le \Devi(\check{\w}, \w)$ be the index such that $(\check{Z}_{m} o, \mathbf{Y}_{i}(\w), Z_{n'} o)$ is $D_{0}$-semi-aligned for all $n' \ge \Devi(\check{\w}, \w)$ and $m \ge 0$. Since $(\check{Z}_{m} o, \mathbf{Y}_{i}(\w), Z_{n} o)$ is $D_{0}$-semi-aligned, Corollary \ref{cor:semiAlign} tells us the geodesic $[\check{Z}_{m} o, Z_{n} o]$ contains a point $q$ in the $0.1E_{0}$-$d^{sym}$-neighborhood of $Z_{i-M_{0}} o$, the beginning point of $\mathbf{Y}_{i}(\w)$. We then have \[\begin{aligned}
d(\check{Z}_{m} o, Z_{n} o) & = d(\check{Z}_{m} o, q) + d(q, Z_{n} o)\\
&\ge \big(d(\check{Z}_{m} o, Z_{i-M_{0}} o) - d(q, Z_{i-M_{0}} o) \big) + \big(d(Z_{i-M_{0}} o, Z_{n} o) - d(Z_{i-M_{0}} o, q) \big) \\
&\ge d(\check{Z}_{m} o, o) - d(Z_{i-M_{0}} o, o) + d(o, Z_{n} o) - d(o, Z_{i-M_{0}} o) - 0.1E_{0}.
\end{aligned}
\]
This implies that \begin{equation}\label{eqn:GromEstimDevi1}
(\check{Z}_{m} o, Z_{n} o)_{o} \le \frac{1}{2} \left( d(Z_{i-M_{0}} o, o) + d(o, Z_{i-M_{0}} o) \right) + 0.05E_{0}.
\end{equation} Meanwhile, $(o, \mathbf{Y}_{i}(\w), Z_{k} o)$ and $(Z_{k} o, \bar{\mathbf{Y}}_{i} (\w), o)$ are both $D_{0}$-semi-aligned as well. By Corollary \ref{cor:semiAlign}, $[o, Z_{k} o]$ contains a subsegment longer than $100E_{0}$ that is $0.1E_{0}$-fellow traveling with $\mathbf{Y}_{i}(\w)$. It follows that $d(o, Z_{i-M_{0}} o) \le d(o, Z_{k} o) - 99E_{0}$. $(\ast$) Similarly, by applying Corollary \ref{cor:semiAlign} to $[Z_{k} o, o]$, we deduce that $d(Z_{i-M_{0}} o, o) \le d(Z_{k} o, o) - 99E_{0}$. ($\ast\ast$) By combining these with Inequality \ref{eqn:GromEstimDevi1}, we conclude the first claim.

Let us now see the second assertion. When $m \le \check{\Devi}(\check{\w}, w)$ and $n \le \Devi(\check{\w}, \w)$, the conclusion follows from the definition of the Gromov product.

Next, consider the case that $m \ge \check{\Devi}$ and $n \ge \Devi$. Together with the choice of $i \le \Devi(\check{\w}, \w)$, let $j \le \check{\Devi}(\check{\w}, \w)$ be the index such that $(\check{Z}_{n''} o, \bar{\mathbf{Y}}_{j}(\check{\w}), Z_{m} o)$ is $D_{0}$-semi-aligned for all $n'' \ge \check{\Devi}(\check{\w}, \w)$ and $m \ge 0$. Then $(\check{Z}_{n''} o, \bar{\mathbf{Y}}_{j}(\check{\w}), Z_{m} o)$ is $D_{0}$-semi-aligned, hence $D_{1}$-aligned by Corollary \ref{cor:semiAlign},  for $n'' \ge \check{\Devi}$ and $m = i-M_{0}, \ldots, i$. Similarly, $(\check{Z}_{m} o, \mathbf{Y}_{i}(\w), Z_{n'}o)$ is $D_{1}$-aligned for $n' \ge \Devi$ and $m = j-M_{0}, \ldots, j$. We conclude that \[
\big(\check{Z}_{m}o, \bar{\mathbf{Y}}_{j} (\check{\w}), \mathbf{Y}_{i} (\w), Z_{n} o\big) \,\,\textrm{is $D_{1}$-aligned}.
\]
By Proposition \ref{prop:BGIPWitness}, there exist $p, q \in [\check{Z}_{m} o, Z_{n} o]$, with $p$ coming first, such that $p \in \mathcal{N}_{0.1E_{0}} (\check{Z}_{j-M_{0}} o)$ and $q \in \mathcal{N}_{0.1E_{0}} (Z_{i-M_{0}} o)$. We then have \[\begin{aligned}
d(\check{Z}_{m} o, Z_{n} o) &= d(\check{Z}_{m} o,p) + d(p, q) + d(q, Z_{n} o)\\
&\ge d(\check{Z}_{m} o, p) + d(q, Z_{n} o) \\
&\ge d(\check{Z}_{m} o, \check{Z}_{j-M_{0}} o)+ d(Z_{i-M_{0}} o, Z_{n} o) - 0.2E_{0} \\
&\ge d(\check{Z}_{m} o, o) + d(o, Z_{n} o) - d(\check{Z}_{j-M_{0}} o, o) - d(o, Z_{i-M_{0}} o) - 0.2E_{0}.
\end{aligned}
\]
Meanwhile, by $(\ast$) and $(\ast\ast$), we have $ d(\check{Z}_{j-M_{0}} o, o) \le  d(\check{Z}_{l} o, o) + 99E_{0}$ and $d(o, Z_{i-M_{0}} o) \le d(o, Z_{k} o)+ 99E_{0}$. Combining these yields the conclusion.

Finally, consider the case that $m \ge \check{\Devi}$ and $n \le \Devi$. Then $(\check{Z}_{m} o, \bar{\mathbf{Y}}_{j}(\check{w}), Z_{n} o)$ is $D_{1}$-aligned. By Proposition \ref{prop:BGIPWitness}, there exist $p\in [\check{Z}_{m} o, Z_{n} o]$ such that $p \in \mathcal{N}_{0.1E_{0}} (\check{Z}_{j-M_{0}} o)$. We then have \[\begin{aligned}
d(\check{Z}_{m} o, Z_{n} o) &\ge d(\check{Z}_{m} o,p) \ge d(\check{Z}_{m} o, o) - d(\check{Z}_{j-M_{0}} o, o)  - 0.05E_{0}.
\end{aligned}
\]
This implies that $(\check{Z}_{m} o, Z_{n} o)_{o} \le \frac{1}{2} \left[ d(\check{Z}_{j-M_{0}} o, o) + d(o, Z_{n} o) \right]$. Meanwhile, by $(\ast\ast)$, we have $d(\check{Z}_{j-M_{0}} o, o) \le  d(\check{Z}_{l} o, o) + 99E_{0}$. Combining these results yields the conclusion. The case $m \le \check{\Devi}$ and $n \ge \Devi$ can be handled with a similar argument.
\end{proof}

Because Inequality \ref{eqn:GromEstimDeviMod1} involves $d^{sym}(o, Z_{k} o)$ instead of $d(o, Z_{k} o)$, \cite[Proposition 4.10]{choi2022random2} now requires the finite $p$-th moment of $\mu$ with respect to $d^{sym}$, or equivalently, the finite $p$-th moments of $\mu$ and $\check{\mu}$: 

\begin{prop}[{\cite[Proposition 4.10]{choi2022random1}}] \label{prop:Devi}
Let $(X, G, o)$ be as in Convention \ref{conv:main}, let $p>0$ and let $((\check{Z}_{n})_{n>0}, (Z_{n})_{n>0})$ be the (bi-directional) random walk generated by a non-elementary probability measure $\mu$ on $G$. Suppose that both $\mu$ and $\check{\mu}$ has finite $p$-th moment. Then the random variable $\sup_{n, m \ge 0} (\check{Z}_{m} o, Z_{n} o)_{o}$ has finite $2p$-th moment.
\end{prop}

\begin{proof}
Note that the assumption implies that $\E_{\mu} [d^{sym}(o, go)^{p}] < +\infty$. The proof is done almost verbatim to the proof of \cite[Proposition 4.10]{choi2022random1}, after replacing $d( \cdot, \cdot)$ with $\frac{1}{2} d^{sym}(\cdot, \cdot)$ throughout.
\end{proof}

By using Inequality \ref{eqn:GromEstimDeviMod2} instead, we obtain the following weaker result:

 \begin{prop}[{\cite[Proposition 4.10]{choi2022random1}}] \label{prop:DeviWeak}
Let $(X, G, o)$ be as in Convention \ref{conv:main}, let $p>0$ and let $((\check{Z}_{n})_{n>0}, (Z_{n})_{n>0})$ be the (bi-directional) random walk generated by a non-elementary probability measure $\mu$ on $G$ with finite $p$-th moment. Then the random variable $\sup_{n, m \ge 0} (\check{Z}_{m} o, Z_{n} o)_{o}$ has finite $p$-th moment.
\end{prop}

\begin{proof}
Let $K'$ be as in Lemma \ref{lem:Devi}. We define $D_{k} := \sum_{i=1}^{k} d(o, g_{i} o)$ and $\check{D}_{k} := \sum_{i=1}^{k} d(\check{g}_{i} o, o) = \sum_{i=1}^{k} d(o, g_{1-i} o)$. Note the following consequence of the triangle inequality:\[
d(o, Z_{1} o), \ldots, d(o, Z_{k} o) \le D_{k}, \quad d(\check{Z}_{1} o, o), \ldots, d(\check{Z}_{k} o, o) \le \check{D}_{k}.
\]
Also note the following computation: if non-negative constants $A, B$ and $C$ satisfy $A \le \frac{1}{2}(B+C)$, then \[
A^{p} \le \left( \frac{1}{2} (B+C) \right)^{p} \le \left( \max(B, C) \right)^{p} \le B^{p} + C^{p}.
\]
Noting this inequality, by Lemma \ref{lem:GromEstimDevi}, we have \[\begin{aligned}
\sup_{m, n \ge 0} (\check{Z}_{m} o, Z_{n} o)_{o}^{p} &\le \sup_{m, n \ge 0} \left( d(\check{Z}_{\min(m, \check{\Devi})} o, o)^{p} + d(o, Z_{\min(n, \Devi)} o)^{p} \right) \\
&\le D_{\Devi}^{p} + \check{D}_{\check{\Devi}}^{p}.
\end{aligned}
\]

Hence, it suffices to control $\E[ D_{\Devi}^{p}]$ and $\E[\check{D}_{\check{\Devi}}^{p}]$ on $(G^{\Z_{>0}}, \check{\mu}^{\Z_{>0}}) \times (G^{\Z_{>0}}, \mu^{\Z_{>0}})$. The computations are analogous so we will only show that $\E[D_{\Devi}^{p}]$ is finite. First, when $p \le 1$, the concavity of $f(t) = t^{p}$ tells us that $D_{k+1}^{p} - D_{k}^{p} \le (D_{k+1} - D_{k})^{p} = d(o, g_{k+1}o)^{p}$ for each $k$. Hence, we compute \begin{equation}\label{eqn:workDevi1}
\begin{aligned}
\E[D_{\Devi}^{p}] &= \E \left[ \sum_{i=0}^{\infty} \left( D_{i+1}^{p} - D_{i}^{p}\right) 1_{i < \Devi}\right] \le \sum_{i=0}^{\infty} \E \left[d(o, g_{i+1} o)^{p} 1_{i < \Devi} \right] \\
&\le \sum_{i=0}^{\infty} \E \left[ d( o, g_{i+1} o)^{p} \cdot \Prob\left[ i < \Devi \, \big|\, g_{i+1} \right] \right] \le \sum_{i=0}^{\infty} \E \left[ d(o, g_{i+1} o)^{p} \cdot K' e^{-i/ K'} \right] & (\because \textrm{Lemma \ref{lem:Devi}}) \\
&\le  \frac{K'}{1-e^{-1/K'}} \E_{\mu}[d(o, go)^{p}] < + \infty.
\end{aligned}
\end{equation}
When $p \ge 1$, we have \[
D_{k+1}^{p} - D_{k}^{p} \le 2^{p} d(o, g_{k+1} o)^{p} + 2^{p} D_{k}^{p-1} d(o, g_{k+1} o).
\]
Using this, we can bound $\E[D_{\Devi}^{p}] = \E \left[ \sum_{i=0}^{\infty} \left( D_{i+1}^{p} - D_{i}^{p}\right) 1_{i < \Devi}\right]$ by a linear combination of $\sum_{i=0}^{\infty} \E \left[d(o, g_{i+1} o)^{p} 1_{i < \Devi} \right] $ and $\sum_{i=0}^{\infty} \E \left[D_{i}^{p-1}d(o, g_{i+1} o) 1_{i < \Devi} \right]$. The former summation was bounded by a multiple of $\E_{\mu}[d(o, go)^{p}]$ in Display \ref{eqn:workDevi1}. We now treat the summands of the latter summation as follows: \begin{equation}\label{eqn:workDevi2}\begin{aligned}
\E \left[D_{i}^{p-1}d(o, g_{i+1} o) 1_{i < \Devi} \right] &\le \E \left[ e^{i/2K'} \cdot  d(o, g_{i+1} o) 1_{i < \Devi} \right] + \E \left[ D_{i}^{p-1} 1_{D_{i} > e^{i/2K'(p-1)}} \cdot d(o, g_{i+1} o) 1_{i < \Devi} \right] \\
&\le e^{i/2K'} \cdot \E\left[ d(o, g_{i+1} o) \Prob \left[i < \Devi \, \big| \, g_{i+1} \right] \right] + \E\left[D_{i}^{p} \cdot e^{-i/2K' (p-1)} \cdot d(o, g_{i+1} o) \right] \\
&\le K' e^{-i/2K'} \E[d(o, g_{i+1} o)^{p}] \\
&+ e^{-i/2K'(p-1)} \cdot \E[d(o, g_{i+1} o)] \cdot i^{p}  \E[d(o, g_{1}o)^{p} + \ldots + d(o, g_{i} o)^{p}].
\end{aligned}
\end{equation}
Here, the final term due to the fact that $\E[(X_{1} + \ldots + X_{n})^{p}] \le \E[ (n \max_{i} X_{i} )^{p} ] = n^{p} \cdot \E[ \max_{i} X_{i}^{p} ] \le n^{p} \E[ X_{1}^{p} + \ldots + X_{n}^{p}]$ for nonnegative RVs $X_{i}$'s. Since this summand is exponentially decaying, its summation is finite as desired.
\end{proof}

Using Inequality \ref{eqn:GromEstimDeviMod2}, we also obtain the exponential deviation inequality as below.

\begin{prop}[{\cite[Corollary 4.12]{choi2022random1}}]\label{prop:DeviExp}
Let $(X, G, o)$ be as in Convention \ref{conv:main} and let $((\check{Z}_{n})_{n>0}, (Z_{n})_{n>0})$ be the (bi-directional) random walk generated by a non-elementary random walk $\mu$ on $G$ with finite exponential moment. Then there exists $K>0$ such that \[
\E\left[ \operatorname{exp}\left(\sup_{n, m \ge 0} (\check{Z}_{m} o, Z_{n} o)_{o} / K \right) \right] < K.
\]
\end{prop}

\begin{proof}
As in the proof of Proposition \ref{prop:DeviWeak}, we let $D_{k} := \sum_{i=1}^{k} d(o, g_{i} o)$ and $\check{D}_{k} := \sum_{i=1}^{k} d(\check{g}_{i} o, o) = \sum_{i=1}^{k} d(o, g_{1-i} o)$. Then by Lemma \ref{lem:GromEstimDevi}, we have \[\begin{aligned}
\operatorname{exp} \left(\sup_{m, n \ge 0} (\check{Z}_{m} o, Z_{n} o)_{o}/K \right) &\le \sup_{m, n \ge 0} \left( \operatorname{exp} \big(d(\check{Z}_{\min(m, \check{\Devi})} o, o)/K\big) + \operatorname{exp}\big(d(o, Z_{\min(n, \Devi)} o)/K \big) \right) \\
&\le \operatorname{exp}(D_{\Devi}/K) + \operatorname{exp}(\check{D}_{\check{\Devi}}/K).
\end{aligned}
\]
Also note that \[
\E[\operatorname{exp}(D_{\Devi}/K)] \le \sum_{i=0}^{\infty} \E\left[\operatorname{exp}(D_{i+1}/K) 1_{i < \Devi } \right], \quad 
\E[\operatorname{exp}(\check{D}_{\check{\Devi}}/K)] \le \sum_{i=0}^{\infty} \E\left[\operatorname{exp}(\check{D}_{i+1}/K) 1_{i < \check{\Devi} } \right].
\]
Hence, it suffices to prove that $ \E\left[\operatorname{exp}(D_{i+1}/K) 1_{i < \Devi } \right]$ (and its symmetric counterpart) are exponentially summable for large enough $K$. This follows from  the proof of \cite[Corollary 4.12]{choi2022random1}.
\end{proof}

From these deviation inequalities, we deduced several limit theorems in \cite[Section 4.4]{choi2022random1}. These include:

\begin{thm}[{\cite[Theorem 4.13]{choi2022random1}}]\label{thm:CLTStrong}
Let $(X, G, o)$ be as in Convention \ref{conv:main} and let $(Z_{n})_{n}$ be the random walk generated by a non-elementary probability measure $\mu$ on $G$ with finite second moment. Then the following limit (called the \emph{asymptotic variance of $\mu$}) exists: \[
\sigma^{2} (\mu) := \lim_{n \rightarrow \infty} \frac{1}{n} Var[d(o, Z_{n} o)],
\]
and the random variable $\frac{1}{\sqrt{n}} [d(o, Z_{n}o) - \lambda(\mu) n]$ converges in law to the Gaussian law $\mathcal{N}(0, \sigma(\mu))$ with zero mean and variance $\sigma^{2}(\mu)$.
\end{thm}

\begin{thm}[{\cite[Theorem 4.16]{choi2022random1}}]\label{thm:LILStrong}
Let $(X, G, o)$ be as in Convention \ref{conv:main} and let $(Z_{n})_{n}$ be the random walk generated by a non-elementary probability measure $\mu$ on $G$ with finite second moment. Then for almost every sample path $(Z_{n})_{n}$ we have \[
\limsup_{n \rightarrow \infty} \frac{d(o, Z_{n} o) - \lambda(\mu) n}{\sqrt{2n \log \log n}} = \sigma(\mu),
\] 
where $\lambda(\mu)$ is the drift of $\mu$ and $\sigma^{2}(\mu)$ is the asymptotic variance of $\mu$.
\end{thm}

\begin{thm}[{\cite[Theorem 4.18]{choi2022random2}}]\label{thm:tracking}
Let $(X, G, o)$ be as in Convention \ref{conv:main} and let $(Z_{n})_{n}$ be the random walk generated by a non-elementary measure $\mu$ on $G$. \begin{enumerate}
\item Suppose that  $\mu$ and $\check{\mu}$ have finite $p$-th moment for some $p > 0$. Then for almost every sample path $(Z_{n}(\w))_{n \ge 0}$, there exists a quasigeodesic $\gamma$ such that \[
\lim_{n \rightarrow\infty} \frac{1}{n^{1/2p}} d^{sym}(Z_{n} o, \gamma) = 0.
\]
\item Suppose that $\mu$ and $\check{\mu}$ have finite exponential moment. Then there exists $K<+\infty$ such that for almost every sample path $(Z_{n}(\w)_{n\ge 0}$, there exists a quasigeodesic $\gamma$ satisfying \[
\lim_{n \rightarrow \infty} \frac{1}{\log n} d^{sym}(Z_{n} o, \gamma) < K.
\]
\end{enumerate}
\end{thm}

In \cite[Section 6]{choi2022random1}, we discussed a more complicated version of pivotal times following \cite[Section 5.A]{gouezel2022exponential}. The ingredients required for the construction of pivotal times are again Definition \ref{dfn:Schottky} and Lemma \ref{lem:1segment}. Furthermore, \cite[Lemma 3.17]{choi2022random1} linked the alignment and the almost additivity of (forward) step progresses, and Proposition \ref{prop:BGIPWitness} now plays the same role. Hence, by following \cite[Section 6]{choi2022random1}, we deduce the large deviation principle:

\begin{thm}[{\cite[Theorem 6.4]{choi2022random1}}]\label{thm:LDPStrong}
Let $(X, G, o)$ be as in Convention \ref{conv:main} and let $(Z_{n})_{n}$ be the random walk generated by a non-elementary probability measure $\mu$ on $G$. Let $\lambda(\mu) = \lim_{n} \frac{1}{n} \E[d(o, Z_{n} o)]$ be the drift of $\mu$. Then for each $0 < L <\lambda(\mu)$, the probability $\Prob(d(o, Z_{n} o) \le Ln)$ decays exponentially as $n$ tends to infinity.
\end{thm}

\begin{cor}[{\cite[Corollary 6.5]{choi2022random1}}] \label{cor:LDPMod}
Let $(X, G, o)$ be as in Convention \ref{conv:main} and let $(Z_{n})_{n\ge0}$ be the random walk generated by a non-elementary probability measure $\mu$ on $G$. Then there exists a proper convex function $I : \mathbb{R} \rightarrow [0, +\infty]$, vanishing only at the drift $\lambda(\mu)$, such that  \[\begin{aligned}
- \inf_{x \in \operatorname{int}(E)} I(x) &\le \liminf_{n \rightarrow \infty} \frac{1}{n} \log \Prob\left( \frac{1}{n}d(id, Z_{n}) \in E\right), \\
 -\inf_{x \in \bar{E}} I(x) & \ge\limsup_{n \rightarrow \infty} \frac{1}{n} \log \Prob\left(\frac{1}{n} d(id, Z_{n}) \in E\right) 
\end{aligned}
\]
holds for every measurable set $E \subseteq \mathbb{R}$.
\end{cor}

\subsection{Results from \cite{choi2022random2}} \label{subsection:limit2}

We now discuss the results of \cite{choi2022random2}, which consists of two parts. In the first part, we made use of the properties of the indices $\Devi(\check{\w}, \w)$ and $\check{\Devi}(\check{\w}, \w)$ in Definition \ref{dfn:Devi}. As we saw in Lemma \ref{lem:GromEstimDevi}, the asymmetry of the metric necessitates suitable modification. 

For us, \cite[Claim 3.3]{choi2022random2} during the proof of \cite[Theorem B]{choi2022random2} requires a modification. This claim asserts that the discrepancy between the displacement $d(o, Z_{n} o)$ and the translation length $\tau(Z_{n})$ of a random isometry $Z_{n}$ is bounded by (a multiple of) the minimum of $d(o, Z_{\nu} o)$ and $d(o, \check{Z}_{\check{\nu}} o)$. This remains true if we replace $d(\cdot, \cdot)$ with $d^{sym}(\cdot, \cdot)$, and we obtain the full conclusion of \cite[Theorem B.(1)]{choi2022random2} if $\mu$ and $\check{\mu}$ have finite $p$-th moment. 

In general, we have the following modified version of \cite[Claim 3.3]{choi2022random2}.

\begin{lem}
Let $(X, G, o)$ be as in Convention \ref{conv:main}, let $\mu$ be a non-elementary probability measure on $G$ with finite $p$-th moment, and let $S$ be a fairly long Schottky set for $\mu$. Then for each $n>0$, there exist RVs $D_{n}$ and $\check{D}_{n}$ such that the following holds: \begin{enumerate}
\item $D_{n}$ and $\check{D}_{n}$ each have the same distribution with $d(o, Z_{\Devi}(\w) o)$ and $d(Z_{\check{\Devi}}(\check{\w}) o)$ for $(\check{\w} ,\w) \in (G^{\Z_{>0}}, \check{\mu}^{\Z_{>0}}) \times (G^{\Z_{>0}}, \mu^{\Z_{>0}})$;
\item $d(o, Z_{n} o) - \tau(Z_{n}) \le \check{D}_{n} + D_{n}$ holds outside a set of exponentially decaying probability (in $n$).
\end{enumerate}
\end{lem}

\begin{proof}
We follow the proof of \cite[Theorem B]{choi2022random2}. Let $K_{0}, D_{0}, D_{1}, E_{0}$ be the constants associated to the Schottky set $S$.

First, we prepare two bi-infinite sequences $(g_{i})_{i \in Z}$ and $(h_{i})_{i \in \Z}$ of independent RVs, all distributed according to $\mu$. Following our standard convention, $Z_{i}$ always denotes $g_{1} \cdots g_{i}$. We now fix $n$, and define \[\begin{aligned}
&g_{i; 0} := \left\{ \begin{array}{cc} g_{i} & i = 1, \ldots, \lfloor n/2 \rfloor \\ h_{i} & i > \lfloor n/2  \rfloor, \end{array}\right. &\check{g}_{i; 0} := \left\{ \begin{array}{cc} g_{n - i + 1}^{-1} & i = 1, \ldots, n- \lfloor n/2 \rfloor \\ h_{-i}^{-1} & i >n-\lfloor n/2 \rfloor, \end{array}\right. \\
&g_{i; 1} := \left\{ \begin{array}{cc} g_{ \lfloor n/2 \rfloor +i} & i = 1, \ldots, n-\lfloor n/2 \rfloor \\ h_{i}  &i > n- \lfloor n/2  \rfloor, \end{array}\right.&\check{g}_{i; 1} := \left\{ \begin{array}{cc} g_{\lfloor n/2 \rfloor  - i + 1}^{-1} & i = 1, \ldots, \lfloor n/2 \rfloor \\ h_{-i}^{-1} & i > \lfloor n/2  \rfloor .\end{array}\right. 
\end{aligned}
\]

Then $\big((\check{g}_{i; t})_{i > 0}, (g_{i; t})_{i>0} \big)$ is distributed according to $\check{\mu}^{\Z_{>0}} \times \mu^{\Z_{>0}}$ for $t=0, 1$. Using them, we similarly defined other RVs such as  \[\begin{aligned}
Z_{i; t} &:= g_{1; t} \cdots g_{i; t}, &\check{Z}_{i; t} &:= \check{g}_{1; t} \cdots \check{g}_{i;t}, \\
\Devi_{t} &:= \Devi\big( (\check{g}_{i;t})_{i>0}, (g_{i; t})_{i>0}\big), &\check{\Devi}_{t} &:= \check{\Devi}\big(  (\check{g}_{i;t})_{i>0}, (g_{i; t})_{i>0}\big).
\end{aligned}
\]
Finally, we will set $D_{n} := d(o, Z_{\Devi_{0}; 0} o)$ and $\check{D}_{n}:= d(\check{Z}_{\check{\Devi}_{0};0} o, o)$. Since $\big((\check{g}_{i; 0})_{i > 0}, (g_{i; 0})_{i>0} \big)$  is distributed according to $\check{\mu}^{\Z_{>0}} \times \mu^{\Z_{>0}}$, we can verify Item (1) of the conclusion.

We then define \[
A_{n} := \Big\{ \max \{\Devi_{0}, \check{\Devi}_{0}, \Devi_{1}, \check{\Devi}_{1} \} \ge n/10\Big\}.
\]
Since $\big((\check{g}_{i; 0})_{i > 0}, (g_{i; 0})_{i>0} \big)$ and $\big((\check{g}_{i; 1})_{i > 0}, (g_{i; 1})_{i>0} \big)$ are both distributed according to $\check{\mu}^{\Z_{>0}} \times \mu^{\Z_{>0}}$, Lemma \ref{lem:Devi} implies that $\Prob(A_{n})$ decays exponentially in $n$.

In the proof of \cite[Claim 3.3]{choi2022random2}, using the alignment lemma (Proposition \ref{prop:BGIPWitness} in the current setting), we showed the existence of $0 \le i, j \le n/2$ such that: \begin{enumerate}
\item $\gamma := (Z_{i-M_{0}} o, \ldots, Z_{i} o)$ and $\check{\gamma} := (Z_{n-j}o, \ldots, Z_{n-j + M_{0}} o)$ are Schottky axes;
\item on $A_{n}$, $(o, \,\gamma, Z_{\Devi_{0}; 0} \,o)$ is $D_{1}$-aligned;
\item on $A_{n}$, $(\check{Z}_{\check{\Devi}_{0}; 0} o,\, Z_{n}^{-1} \check{\gamma}, \,o)$ is $D_{1}$-aligned;
\item on $A_{n}$, for each $k > 0$, there exist points $p_{0}, q_{0}, \ldots, p_{k-1}, q_{k-1}$  on $[o, Z_{n}^{k} o]$, from left to right, so that \[
d^{sym}(p_{i}, Z_{n}^{i} \cdot Z_{i-M_{0}} o) < 0.1E_{0}, \quad d^{sym}(q_{i}, Z_{n}^{i} \cdot Z_{n-j + M_{0}} o) < 0.1E_{0}.
\]
\end{enumerate}
Then we have \begin{align}
d(p_{i}, q_{i}) &\ge d(Z_{i-M_{0}} o, Z_{n-j + M_{0}} o) - 0.2E_{0} \label{eqn:critLDPEqn}\\
&\ge d(o, Z_{n} o) - d(o, Z_{i-M_{0}} o) - d(Z_{n}^{-1} Z_{n-j + M_{0}} o, o) - 0.2E_{0},
\end{align}
and we deduce \[
d(o, Z_{n}^{k} o) \ge \sum_{i=1}^{k} d(p_{i-1}, q_{i-1})  \ge k \cdot \left( d(o, Z_{n} o) - d(o, Z_{i-M_{0}} o) - d(Z_{n}^{-1} Z_{n-j + M_{0}} o, o) - 0.2E_{0}\right).
\]
By dividing by $k$ and taking the limit, we deduce that \[
d(o, Z_{n} o) - \tau(Z_{n}) \le d(o, Z_{i-M_{0}} o) + d(Z_{n}^{-1} Z_{n-j+ M_{0}} o, o) + 0.2E_{0}.
\]
Meanwhile, the alignments of $(o, \,\gamma, Z_{\Devi_{0}; 0} \,o)$ and $(\check{Z}_{\check{\Devi}_{0}; 0} o,\, Z_{n}^{-1} \check{\gamma}, \,o)$ imply that $d(o, Z_{i-M_{0}} o) + 0.1E_{0}$ is smaller than $D_{n} := d(o, Z_{\Devi_{0}; 0} o)$, and that $d(Z_{n}^{-1} Z_{n-j+ M_{0}} o, o) + 0.1E_{0}$ is smaller than $\check{D}_{n} := d(\check{Z}_{\check{\Devi}_{0}; 0} o, o)$. This concludes Item (2).
\end{proof}

Given this lemma, it follows that \[
\Prob( d(o, Z_{n} o) - \tau(Z_{n}) \ge \epsilon n^{1/p} ) \le \Prob(A_{n}^{c}) + \Prob( D_{n} + \check{D}_{n} \ge \epsilon n^{1/p})
\]for every $\epsilon > 0$. Since $D_{n}+ \check{D}_{n}$ has the same distribution with $d( Z_{\check{\Devi}}(\check{\w}) o, o) + d(o, Z_{\Devi}(\w) o)$, which has finite $p$-th moment (cf. Proposition \ref{prop:DeviWeak}), $\Prob(D_{n} + \check{D}_{n} \ge \epsilon n^{1/p})$ is summable. Using Borel-Cantelli lemma, we conclude:

\begin{thm}\label{thm:discrepancy}
Let $(X, G, o)$ be as in Convention \ref{conv:main}, and $(Z_{n})_{n>0}$ be the random walk generated by a non-elementary measure $\mu$ on $G$. If $\mu$ has finite $p$-th moment for some $p > 0$, then \[
\lim_{n \rightarrow \infty} \frac{1}{n^{1/p}} [d(o, Z_{n} o)- \tau(Z_{n}) ]= 0 \quad \textrm{a.s.}
\]
\end{thm}

We note again that if $\mu$ and $\check{\mu}$ both have finite $p$-th moment, the proof of \cite[Theorem B]{choi2022random1} works up to changing $d(\cdot, \cdot)$ with $d^{sym}(\cdot, \cdot)$ and give the stronger result ($o(n^{1/2p})$-tracking, and logarithmic tracking when $p = 1$).

Although we have weaker result than \cite[Theorem B]{choi2022random1}, Theorem \ref{thm:discrepancy} is sufficient to deduce the corollaries. Namely, combined with the SLLN and CLT for displacement (Theorem \ref{thm:LDPStrong} and \ref{thm:CLTStrong}), Theorem \ref{thm:discrepancy} implies: 

\begin{cor}[{\cite[Corollary 3.7]{choi2022random2}}]\label{cor:SLLNFinite}
Let $(X, G, o)$ be as in Convention \ref{conv:main}, and let $(Z_{n})_{n\ge 0}$ be the random walk generated by a non-elementary measure $\mu$ on $G$ with finite first moment. Then \begin{equation}\label{eqn:SLLN}
\lim_{n} \frac{1}{n} \tau(Z_{n}) = \lambda
\end{equation}
holds almost surely, where $\lambda = \lambda(\mu)$ is the escape rate of $\mu$.
\end{cor}

\begin{cor}[{\cite[Corollary 3.8]{choi2022random2}}]\label{cor:CLT}
Let $(X, G, o)$ be as in Convention \ref{conv:main}, and let $(Z_{n})_{n\ge 0}$ be the random walk generated by a non-elementary measure $\mu$ on $G$. If $\mu$ has finite second moment, then there exists $\sigma(\mu) \ge 0$ such that $\frac{1}{\sqrt{n}}(\tau(Z_{n}) - n \lambda)$ and $\frac{1}{\sqrt{n}}(d(o, Z_{n} o) - n\lambda)$ converge to the same Gaussian distribution $\mathscr{N}(0, \sigma(\mu)^{2})$ in law. We also have \[
\limsup_{n \rightarrow \infty}  \frac{\tau(Z_{n}) - \lambda n}{\sqrt{2n \log \log n}} =\sigma(\mu) \quad \textrm{almost surely}.
\]
\end{cor}

Meanwhile, the exponential bound in \cite[Theorem A]{choi2022random2} was established based on Inequality \ref{eqn:critLDPEqn} and the exponential bound for displacement (Theorem \ref{thm:LDPStrong}). Moreover, we know that exponentially generic isometries have BGIP, this time using Lemma \ref{lem:BGIPConcat} instead of \cite[Lemma 2.7]{choi2022random2}. Considering these, the proof of \cite[Theorem A]{choi2022random2}  yields the following: 

\begin{thm} \label{thm:expBd}
Let $(X, G, o)$ be as in Convention \ref{conv:main} and let $(Z_{n})_{n \ge 0}$ be the random walk generated by a non-elementary measure $\mu$ on $G$. Let $\lambda(\mu)$ be the escape rate of $\mu$, i.e., $\lambda(\mu) := \lim_{n \rightarrow \infty} \E_{\mu^{\ast n}} [d(o, go)] / n$. Then for each $0 < L < \lambda(\mu)$, there exists $K>0$ such that for each $n$ we have  \[
\Prob\Big(\textrm{$Z_{n}$ has BGIP and $\tau(Z_{n}) \ge Ln$}\Big) \ge 1-K e^{-n/K}.
\]
\end{thm}

In the second part of \cite{choi2022random2}, we discussed the pivotal time construction following \cite[Subsection 5.A]{gouezel2022exponential}. \cite[Section 4.1]{choi2022random2} describes pivotal times in a discrete model, which applies to the current setting verbatim. We record one definition from \cite[Section 4.1]{choi2022random2}. 

\begin{definition}[{\cite[Defnition 4.2]{choi2022random2}}] \label{dfn:SchottkyTilde}
Let $(X, G, o)$ be as in Convention \ref{conv:main}. Given a fairly long Schottky set $S$, we define \[
\tilde{S} := \big\{ (\beta, \gamma, v) \in S^{2} \times G : \textrm{$\big(\Gamma(\beta), \Pi(\beta) v \Pi(\gamma) o\big)$ and $\big(v^{-1} o, \Gamma(\gamma)\big)$ are $K_{0}$-aligned} \big\}.
\]
\end{definition}

In \cite[Section 4.2]{choi2022random2}, we made several reductions for random walks. In \cite[Section 4.3]{choi2022random2}, we discussed the pivoting for translation length. Thanks to the existence of Schottky set for non-elementary probability measures,  (Proposition \ref{prop:Schottky}), the property of Schottky sets (Definition \ref{dfn:Schottky}) and alignment lemma (Lemma \ref{lem:1segment}, we can bring these to the current setting of Convention \ref{conv:main}.

For latter use, let us record one lemma from \cite[Section 4.2]{choi2022random2}.

\begin{lem}[cf. {\cite[Lemma 4.6]{choi2022random2}}] \label{lem:firstReduction}
Let $(X, G, o)$ be as in Convention \ref{conv:main}, let $\mu$ be a non-elementary probability measure on $G$, let $\mu'$ be a probability measure dominated by a multiple of a convolution of $\mu$ (i.e., there exists $\alpha, k> 0$ such that $\mu'(g) \le \alpha\mu^{\ast k}(g)$ for all $g \in G$) and let $S \subseteq G^{M_{0}}$ be a fairly long $K_{0}$-Schottky set for $\mu$. Then for each $n$ there exist an integer $m(n)$, a probability space $\Omega_{n}$, a measurable subset $B_{n} \subseteq \Omega_{n}$, a measurable partition $\mathcal{Q}_{n}$ of $B_{n}$, and random variables \[\begin{aligned}
Z & \in G,\\
\{w_{i}, i=0, \ldots, m(n)\} &\in G^{m(n)+1}, \\
\{v_{i} : i=1, \ldots, m(n)\} &\in G^{m(n)},\\
 \{\alpha_{i}, \beta_{i}, \gamma_{i}, \delta_{i} : i=1, \ldots, m(n)\} &\in S^{4m(n)}
\end{aligned}
\]
such that the following hold: \begin{enumerate}
\item $\lim_{n \rightarrow +\infty} \Prob(B_{n}) = 1$ and $\lim_{n \rightarrow +\infty} m(n)/n>0$.
\item On each equivalence class $\mathcal{F} \in \mathcal{Q}_{n}$, $(w_{i})_{i=0}^{m(n)}$ are constant and $(\alpha_{i}, \beta_{i}, \gamma_{i}, \delta_{i}, v_{i})_{i=1}^{m(n)}$ are i.i.d.s distributed according to $(\textrm{uniform measure on $S^{4}$}) \times \mu'$.
\item $Z$ is distributed according to $\mu^{\ast n}$ on $\Omega_{n}$ and \[
Z = w_{0} \Pi(\alpha_{1}) \Pi(\beta_{1}) v_{1} \Pi(\gamma_{1}) \Pi(\delta_{1}) w_{1} \cdots \Pi(\alpha_{m(n)}) \Pi(\beta_{m(n)}) v_{m(n)} \Pi(\gamma_{m(n)}) \Pi(\delta_{m(n)}) w_{m(n)}
\]
holds on $A_{n}$.
\end{enumerate}
\end{lem}

In \cite[Lemma 4.6]{choi2022random2}, the lemma was stated for the case $\mu' = \mu$. By employing the decomposition \[
\mu^{2M_{0}} \times\mu^{\ast k} \times \mu^{2M_{0}} = p(\mu_{S}^{2} \times \mu^{\ast k} \times \mu_{S}^{2}) + (1-p) \nu
\]
for a suitable choice of $0<p<1$ and (nonnegative) probability measure $\nu$, the proof of \cite[Lemma 4.6]{choi2022random2} leads to the current general version.

In \cite[Section 4.4]{choi2022random2}, we made the first application of the discussion above, which is the converse of CLT. For future use, we record a slight modification of \cite[Corollary 4.13]{choi2022random2}.

\begin{definition}[{\cite[Definition 4.7]{choi2022random2}}]\label{dfn:wForPivot}
Let $(X, G, o)$ be as in Convention \ref{conv:main}, let $n>0$ and $K>0$, and let $S$ be a fairly long Schottky set. (This determines $\tilde{S}$ following Definition \ref{dfn:SchottkyTilde}.) We say that a sequence $(w_{i})_{i=0}^{n}$ in $G$ is \emph{$K$-pre-aligned} if, for the isometries $W_{0} := w_{0}$ and \[
V_{k} := W_{k} \Gamma(\beta_{k+1}) v_{k+1}, \,\, W_{k+1} := V_{k} \Gamma(\gamma_{k+1}) w_{k+1} \quad (k=0, \ldots, n-1),
\]
the sequence \[
\Big(o, \,W_{0} \Gamma(\beta_{1}),\, V_{0} \Gamma(\gamma_{1}),\, \ldots, \,W_{n-1}\Gamma(\beta_{n}), \,V_{n-1} \Gamma(\gamma_{n}), \,W_{n} o\Big)
\]
is $K$-semi-aligned for any choices of $(\beta_{i}, \gamma_{i}, v_{i}) \in \tilde{S}$ ($i=1, \ldots, n$).

We say that an isometry $\phi \in G$ is \emph{$K$-pre-aligned} if \[
\Big(\Gamma(\gamma'),\,\, \Pi(\gamma') \phi \cdot \Gamma(\beta) \Big)
\]
is $K$-semi-aligned for any choices of $(\beta, \gamma, v), (\beta', \gamma', v') \in \tilde{S}$.
\end{definition}

\begin{cor}[cf. {\cite[Corollary 4.13]{choi2022random2}}]\label{cor:pivotCorCombined}
Let $(X, G, o)$ be as in Convention \ref{conv:main}, let $\mu$ be a non-elementary probability measure on $G$, let $\mu'$ be a probability measure dominated by a multiple of a convolution of $\mu$, and let $S$ be a fairly long $K_{0}$-Schottky set for $\mu$ with cardinality at least $400$. Then for each $N>0$, for each $n$, there exist $m(n) \in \{2^{s} : s > 0\}$, a probability space $\Omega_{n}$ with a measurable subset $A_{n} \subseteq \Omega_{n}$, a measurable partition $\mathcal{P}_{n}$ of $A_{n}$ and RVs \[
\begin{aligned}
Z_{n} &\in G, \\
\{w_{i} : i = 0, \ldots, m(n)\} &\in G^{m(n) + 1}, \\
\{v_{i} : i = 1, \ldots, m(n)\} &\in G^{m(n)}, \\
\{\beta_{i}, \gamma_{i} : i = 1, \ldots, m(n)\} &\in S^{2m(n)}
\end{aligned}
\]
such that the following hold: \begin{enumerate}
\item $\lim_{n \rightarrow +\infty} \Prob(A_{n}) = 1$ and $N < m(n) < 2N$ eventually.
\item On $A_{n}$, $(w_{i})_{i=0}^{m(n)}$ is a $D_{0}$-pre-aligned sequence in $G$ and $w_{m(n)}^{-1} w_{0}$ is a $D_{0}$-pre-aligned isometry.
\item On each equivalence class $\mathcal{E} \in \mathcal{P}_{n}$, $(w_{i})_{i=0}^{m(n)}$ are constant and $(\beta_{i}, \gamma_{i}, v_{i})_{i=1}^{m(n)}$ are i.i.d.s distributed according to the measure $\left(\textrm{uniform measure on $S^{2}$}\right) \times \mu'$ conditioned on $\tilde{S}$.
\item $Z_{n}$ is distributed according to $\mu^{\ast n}$ on $\Omega_{n}$ and \[
Z_{n} = w_{0} \Pi(\beta_{1}) v_{1} \Pi(\gamma_{1}) w_{1} \cdots \Pi(\beta_{m(n)})v_{m(n)} \Pi(\gamma_{m(n)}) w_{m(n)}
\]holds on $A_{n}$.
\end{enumerate}
\end{cor}

The difference between \cite[Corollary 4.13]{choi2022random2} and Corollary \ref{cor:pivotCorCombined} is that: \begin{enumerate}
\item $\mu'$ is a measure dominated by some $\mu^{\ast M'}$ instead of $\mu' = \mu$;
\item  $m(n)$ eventually lies in $[N, 2N]$, rather than growing linearly. 
\end{enumerate}
The first item was addressed in Lemma \ref{lem:firstReduction}. Next, the original proof of \cite[Corollary 4.13]{choi2022random2} (which depends on \cite[Lemma 4.6]{choi2022random2}) can in fact handle the second item, thanks to the following fact: if we have a decomposition \[
\mu_{S}^{2M_{0}} \times \mu^{\ast M'} \times \mu_{S}^{2M_{0}} = p(\mu_{S}^{2} \times \mu' \times \mu_{S}^{2}) +(1-p) \nu \quad(\nu : \textrm{probability measure})
\]
for some $p=\epsilon > 0$, then the decomposition of the same form is possible for all $0 < p < \epsilon$.

As an example, we can plug in $\mu' = \textrm{atom at $g$}$ for $g \in\llangle  \supp \mu\rrangle$. Combining this with Proposition \ref{prop:BGIPWitness}, we obtain the following.

\begin{prop}\label{prop:stability}
Let $(X, G, o)$ be as in Convention \ref{conv:main}. Given  a non-elementary probability measure $\mu$ on $G$, there exists $K>0$ such that the following holds.

For each $g \in \llangle \supp \mu \rrangle$, there exists $\epsilon>0$ such that the following holds outside a set of exponentially decaying probability. For a random path $(Z_{i})_{i=1}^{\infty}$, there exist $0 <i(1) < \ldots < i(\epsilon n) < n$ such that $Z_{i(1)} o, Z_{i(1)} g^{M}o$, $ \ldots$, $Z_{i(\epsilon n)} o,Z_{i(\epsilon n)} g^{M} o$ are $K$-$d^{sym}$-close to points $p_{1}, \ldots, p_{2\epsilon n}$ on $[o, Z_{n} o]$, respectively, such that $d(o, p_{i}) \le d(o, p_{i+1} o)$ for each $i$.
\end{prop}

The discussion so far was about securing the alignment. We now need to couple the alignment and (almost) additivity of the displacement along sample path. Thanks to Proposition \ref{prop:BGIPWitness}, we can generalize such a coupling to the setting of Convention \ref{conv:main} as well.

 \begin{lem}\label{lem:alignedGromov}
Let $(X, G, o)$ be as in Convention \ref{conv:main}, let $S$ be a large and fairly long $K_{0}$-Schottky set for $\mu$, and fix a $D_{0}$-pre-aligned sequence $(w_{i})_{i=0}^{n}$ in $G$ such that $w_{n}^{-1} w_{0}$ is a $D_{0}$-pre-aligned isometry. Given $(\beta_{i}, \gamma_{i}, v_{i}) \in \tilde{S}$ for $i=1, \ldots, n$, define $W_{0}:= w_{0}$ and \[
\begin{aligned}
V_{k} &:= W_{k} \Gamma(\beta_{k+1}) v_{k+1}, \,\, W_{k+1} := V_{k} \Gamma(\gamma_{k+1}) w_{k+1}& (k=0, \ldots, n-1), \\
(x_{2k-1}, x_{2k}) &:= \big(W_{k-1} o, \,\,V_{k-1} \Pi(\gamma_{k}) o\big) & (k=1, \ldots, n), \\
x_{0} &:= 0, \,\,x_{2n+1} := W_{n} o.&
\end{aligned}
\]
Let $\mu'$ be a probability measure on $G$ and let $(\beta_{i}, \gamma_{i}, v_{i})_{i=1}^{n}$ be i.i.d.s distributed according to $(\textrm{uniform measure on $S^{2}$}) \times \mu'$. Then the following hold:\begin{enumerate}
\item $d(x_{2l-1}, x_{2l})$ and $d(o, v_{l} o)$ differ by at most $2 \max \{d(o, \Pi(\alpha) o) : \alpha \in S\}$;
\item $\{d(x_{2l}, x_{2l+1}) : l = 0, \ldots, n\}$ are constant RVs;
\item for each $0 \le i \le j \le k \le 2n+ 1$, $(x_{i}, x_{k})_{x_{j}}$ is bounded by $E_{0}$;
\item for each $0 \le i \le j \le k \le i' \le j' \le k' \le 2n+1$, $(x_{i}, x_{k})_{x_{j}}$ and $(x_{i'}, x_{k'})_{x_{j'}}$ are independent;
\item the translation length $\tau(W_{n})$ and $d(x_{0}, x_{2n})$ differ by at most $E_{0}$.
\end{enumerate}
\end{lem}

Combining Corollary \ref{cor:pivotCorCombined} and Lemma \ref{lem:alignedGromov}, we obtained the following converse of CLT: 

\begin{prop}[{\cite[Proposition 4.16]{choi2022random2}}]\label{prop:CLTConverse}
Let $(X, G, o)$ be as in Convention \ref{conv:main} and let $(Z_{n})_{n \ge 0}$ be the random walk on $G$ generated by a non-elementary measure $\mu$ with infinite second moment. Then for any sequence $(c_{n})_{n}$ of real numbers,  neither of $\frac{1}{\sqrt{n}} (d(o, Z_{n} o) - c_{n})$ and $\frac{1}{\sqrt{n}} (\tau(Z_{n}) - c_{n})$ converges in law.
\end{prop}

The remaining part of \cite{choi2022random2} works verbatim, and we obtain the following results.
\begin{thm}\label{thm:qi} 
Let $(X, G, o)$ be as in Convention \ref{conv:main}, and $(Z_{n}^{(1)}, \ldots, Z_{n}^{(k)})_{n \ge 0}$ be $k$ independent random walks generated by a non-elementary measure $\mu$ on $G$. Then there exists $K>0$ such that \[
\Prob \left[ \langle Z_{n}^{(1)}, \ldots, Z_{n}^{(k)} \rangle \,\,\textrm{is q.i. embedded into a quasi-convex subset of $X$} \right] \ge 1- K e^{-n/K}.
\]
\end{thm}

\begin{thm}\label{thm:qiCount}
Let $G$ be a finitely generated group acting on an asymmetric metric space $X$ with at least two independent BGIP elements. Then for each $k > 0$, there exists a finite generating set $S$ of $G$ such that \[
\frac{\#\left\{(g_{1}, \ldots, g_{k}) \in \big(B_{S}(n) \big)^{k} :\begin{array}{c}  \langle g_{1}, \ldots, g_{k} \rangle \,\, \textrm{is q.i. embedded into} \\ \textrm{a quasi-convex subset of $X$} \end{array} \right\} } {\big(\# B_{S}(n)\big)^{k}}
\]
converges to 1 exponentially fast.
\end{thm}

\subsection{Proof of Theorem \ref{thm:mismatch}} \label{subsection:mismatch}

We first formulate a variation of Proposition \ref{prop:Schottky}.

\begin{lem}\label{lem:SchottkyVar}
Let $(X, G, o)$ be as in Convention \ref{conv:main} and let $\mu$ be a non-elementary probability measure on $G$. Then for each $N, L>0$, there exist $K = K(N)>0$, $n > L$ and a $K$-Schottky set $S$ of cardinality $N$ in $(\supp \mu)^{n}$ satisfying the following: for every $g \in G$, there exist $\alpha, \alpha' \in S$ such that 
\begin{equation}\label{eqn:SchottkyVar}
\Big(\Gamma(\beta), \Pi(\beta) \cdot \big(\Pi(\alpha) g \Pi(\alpha')\big)\cdot \Pi(\gamma) o\big)\,\, \textrm{and}\,\,\Big(\big(\Pi(\alpha) g \Pi(\alpha')\big)^{-1} o, \Gamma(\gamma)\Big)\,\,\textrm{are $K$-aligned} \quad (\forall \beta, \gamma \in S).
\end{equation}
\end{lem}

\begin{proof}
Let $N, L>0$ be given. Without loss of generality, we can assume $N > 10$. We take $K_{0} = K(N)$ be as in Proposition \ref{prop:Schottky}, $D_{0} = D(K_{0}, K_{0})$ be as in Lemma \ref{lem:1segment} and $K = E(D_{0}, K_{0})$, $L' = L(D_{0}, K_{0})$ be as in Proposition \ref{prop:BGIPWitness}. Note that $K$ only depends on $N$ and not on $L$.

Now, by Proposition \ref{prop:Schottky}, there exists a $K_{0}$-Schottky set $S$ of cardinality $N$ in $(\supp \mu)^{n}$ for some $n > L + L' + 2K_{0}^{2}$. We claim that $S$ satisfies the desired property. First, the inequality $K_{0} < D_{0} < K$ tells us that $S$ is automatically a $K$-Schottky set. It remains to check the condition in Display \ref{eqn:SchottkyVar}.

For this, let $g \in G$. Note that $\#S >10$. By the $K_{0}$-Schottky property (2) of $S$, there exist $\alpha' \in S$ such that $(g^{-1}o, \Gamma(\alpha'))$ is $K_{0}$-aligned. Fixing such an $\alpha'$, again the $K_{0}$-Schottky property (2) of $S$ guarantees the existence of $\alpha \in S$ such that $(\Gamma(\alpha), \Pi(\alpha) g o)$ is $K_{0}$-aligned. Then by Lemma \ref{lem:1segment}, $\big(\Gamma(\alpha), \Pi(\alpha)g \Gamma(\alpha') \big)$ is $D_{0}$-aligned.

Next, note that $(\Pi(\alpha)o, \Gamma(\alpha))$ is not $K_{0}$-aligned, as $d(o, \pi_{\Pi(\alpha) o} (\Gamma(\alpha))) = d(o, \Pi(\alpha) o) \ge n/K_{0} - K_{0} > K_{0}$. By Schottky property (2), we conclude that $(\Gamma(\beta), \Pi(\beta)\Pi(\alpha) o)$ is $K_{0}$-aligned for every $\beta \in S \setminus \{\alpha\}$, or equivalently, $(\Gamma(\beta), \Pi(\beta) \Pi(\alpha) o)$ is $K_{0}$-aligned. In this case, since $(\Pi(\beta) o, \Pi(\beta) \Gamma(\alpha) o)$ is $0$-aligned, we conclude that $(\Gamma(\beta), \Pi(\beta) \Gamma(\alpha))$ is $D_{0}$-aligned by Lemma \ref{lem:1segment}. Meanwhile, Schottky property (3) guarantees that $(\Gamma(\alpha), \Pi(\alpha) \Gamma(\alpha))$ is $K_{0}$-aligned as well.

For a similar reason, $(\Gamma(\alpha'), \Pi(\alpha') \Gamma(\gamma))$ is $D_{0}$-aligned for every $\gamma \in S$. In conclusion, the following sequence is $D_{0}$-aligned for every $\beta, \gamma \in S$: \[
\big( \Gamma(\beta), \, \Pi(\beta) \Gamma(\alpha), \, \Pi(\beta) \Pi(\alpha) g \Gamma(\alpha'), \,\Pi(\beta) \Pi(\alpha) g \Pi(\alpha') \Gamma(\gamma) \big).
\]
Also, the involved $K_{0}$-Schottky axes have domains longer than $L'$. We then apply Proposition \ref{prop:BGIPWitness} to conclude that \[
\big( \Pi(\beta)o,\, \Pi(\beta) \Pi(\alpha) g \Pi(\alpha') \Gamma(\gamma) \big), \quad \big( \Gamma(\beta), \, \Pi(\beta) \Pi(\alpha) g \Pi(\alpha') \Pi(\gamma) o \big)
\]
are $K$-aligned as desired.
\end{proof}

We are now ready to prove Theorem \ref{thm:mismatch}.

\begin{proof}
Let $\mu$ be a non-elementary and asymptotically asymmetric measure on $G$. We take $K_{0}= K(400)$ as in Lemma \ref{lem:SchottkyVar}. Lemma \ref{lem:SchottkyVar} guarantees that there exists a fairly long $K_{0}$-Schottky set $S \subseteq (\supp \mu)^{M_{0}}$ with cardinality at least 400, with the property that for every $g \in G$, there exist $\alpha, \alpha' \in S$ such that Display \ref{eqn:SchottkyVar} is satisfied. This choice of $S$ is fixed throughout. 

We now pick a large enough positive integer $n'$. Since $\mu$ is asymptotically asymmetric, we can take $\phi, \varphi \in \supp \mu^{\ast M'}$ for some $M'$ such that \[
\left[ \tau(\phi) - \tau(\phi^{-1}) \right] - \left[ \tau(\varphi) - \tau(\varphi^{-1}) \right] > 0.
\]
By replacing $\phi$ and $\varphi$ with their suitable powers, we may assume that \begin{equation}\label{eqn:thmAE1}
\big(d(o, \phi o)  -d(o, \phi^{-1} o)\big) - \big( d(o, \varphi o) - d(o, \varphi^{-1} o)\big) \ge (6E_{0} + 12K_{0} M_{0} + 2) n'.
\end{equation}

Now let $\alpha, \alpha' \in S$ be as in Display \ref{eqn:SchottkyVar} for $g = \phi$, and define $\phi_{new} = \Pi(\alpha) \phi \Pi(\alpha')$. Then  $\phi_{new}$ belongs to $(\supp \mu^{\ast (M' + 2M_{0})})$, and $(\beta, \gamma, \phi_{new}) \in \tilde{S}$ holds for all $\beta, \gamma \in S$. Note also that \begin{equation}\label{eqn:thmAE2} \begin{aligned}
d(o, \phi o) - d(o, \phi_{new} o) \le d(o, \Pi(\alpha) o) + d(o, \Pi(\alpha')) o) \le 2M_{0} K_{0},\\
d(\phi o, o) - d( \phi_{new} o, o) \le d(\Pi(\alpha) o, o) + d( \Pi(\alpha')) o, o) \le 2M_{0} K_{0}.
\end{aligned}
\end{equation}
and similarly $d(\phi o, o)$ and $d(\phi_{new} o, o)$ differ by $2M_{0} K_{0}$. We can similarly take $\varphi_{new} \in (\supp \mu^{\ast (M' + 2M_{0})})$ such that $(\beta, \gamma, \varphi_{new})\in \tilde{S}$ for all $\beta, \gamma \in S$, and so that \begin{equation}\label{eqn:thmAE3}
\big|d(o, \varphi o) - d(o, \varphi_{new} o) \big| \le 2M_{0}K_{0}, \quad \big|d( \varphi o, o) - d(\varphi_{new} o, o) \big| \le 2M_{0}K_{0}.
\end{equation}
For convenience, we introduce the notation \[
L_{1}^{\pm} := d(o, \phi_{new}^{\pm 1} o),\quad L_{2}^{\pm} := d(o, \varphi_{new}^{\pm 1} o)
\]
Combining Inequality \ref{eqn:thmAE1}, \ref{eqn:thmAE2} and \ref{eqn:thmAE3}, we conclude \begin{equation}\label{eqn:thmAE4}\begin{aligned}
\big(d(o, \phi_{new} o)  -d(o, \phi_{new}^{-1} o)\big) - \big( d(o, \varphi_{new} o) - d(o, \varphi_{new}^{-1} o)\big) &\ge(6E_{0} + 12K_{0} M_{0} + 2) n' - 8M_{0} K_{0}\\
&\ge (6E_{0} + 4K_{0} M_{0} + 2)n'.
\end{aligned}
\end{equation}

Let $\mu'$ be the measure assigning $1/2$ to each of $\phi_{new}$ and $\varphi_{new}$. Then $\mu'$ is dominated by a multiple of $\mu^{\ast (M'+2M_{0})}$. We now apply Corollary \ref{cor:pivotCorCombined}: for each $n$, we obtain an integer $m(n)$, a probability space $\Omega_{n}$ with $A_{n} \subseteq \Omega_{n}$, and a measurable partition $\mathcal{P}_{n}$ of $A_{n}$, and RVs \[
\begin{aligned}
Z_{n} &\in G, \\
\{w_{i} : i = 0, \ldots, m(n)\} &\in G^{m(n) + 1}, \\
\{v_{i} : i = 1, \ldots, m(n)\} &\in G^{m(n)}, \\
\{\beta_{i}, \gamma_{i} : i = 1, \ldots, m(n)\} &\in S^{2m(n)}
\end{aligned}
\]
such that the following hold: \begin{enumerate}
\item $\lim_{n \rightarrow +\infty} \Prob(A_{n}) = 1$ and $m(n) \in [n', 2n']$ eventually;
\item On $A_{n}$, $(w_{i})_{i=0}^{m(n)}$ is a $D_{0}$-pre-aligned sequence in $G$ and $w_{m(n)}^{-1} w_{0}$ is a $D_{0}$-pre-aligned isometry.
\item On each equivalence class $\mathcal{E} \in \mathcal{P}_{n}$, $(w_{i})_{i=0}^{m(n)}$ are constant and $(\beta_{i}, \gamma_{i}, v_{i})_{i=1}^{m(n)}$ are i.i.d.s distributed according to the measure $\left(\textrm{uniform measure on $S^{2}$}\right) \times \mu'$ conditioned on $\tilde{S}$.
\item $Z_{n}$ is distributed according to $\mu^{\ast n}$ on $\Omega_{n}$ and \[
Z_{n} = w_{0} \Pi(\beta_{1}) v_{1} \Pi(\gamma_{1}) w_{1} \cdots \Pi(\beta_{m(n)})v_{m(n)} \Pi(\gamma_{m(n)}) w_{m(n)}
\]holds on $A_{n}$.
\end{enumerate}

Recall that $(\beta, \gamma, g) \in \tilde{S}$ for any $\beta, \gamma \in S$ and $g \in \{\varphi_{new}, \phi_{new}\}$. Hence, the support of $(\textrm{uniform measure on $S^{2}$}) \times \mu'$ restricted to $\tilde{S}$ is $S^{2} \times \{\varphi_{new}, \phi_{new}\}$, and $(\beta_{i}, \gamma_{i}, v_{i})_{i=1}^{m(n)}$ has the law $\big(\textrm{uniform measure on $S^{2}$}\big) \times \mu'$. 

Now, let us pick an equivalence class $\mathcal{E} \in \mathcal{P}_{n}$. Conditioned on $\mathcal{E}$, Lemma \ref{lem:alignedGromov} guarantees the existence of points $o =: x_{0}, x_{1}, \ldots, x_{2m(n)}, x_{2m(n)+1} :=Z_{n} o$ such that the following hold: \begin{enumerate}
\item $d(x_{2l-1}, x_{2l})$ and $d(o, v_{l} o)$ differ by at most $2 K_{0} M_{0}$ for $l = 1, \ldots, m(n)$;
\item $\{d(x_{2l}, x_{2l+1}) : l = 0, \ldots, n\}$ are constant RVs;
\item for each $0 \le i \le j \le k \le 2n+ 1$, $(x_{i}, x_{k})_{x_{j}}$ is bounded by $E_{0}$;
\item the translation length $\tau(W_{n})$ and $d(x_{0}, x_{2n})$ differ by at most $E_{0}$.
\end{enumerate}

In view of Observation \ref{obs:BGIPReversal} and Observation \ref{obs:BGIPAlignRev}(2), we also have that: \begin{enumerate}
\item $d(x_{2l}, x_{2l-1})$ and $d( v_{l} o, o)$ differ by at most $2 K_{0} M_{0}$ for $l = 1, \ldots, m(n)$;
\item $\{d(x_{2l+1}, x_{2l}) : l = 0, \ldots, n\}$ are constant RVs;
\item for each $0 \le i \le j \le k \le 2n+ 1$, $(x_{k}, x_{i})_{x_{j}}$ is bounded by $E_{0}$;
\item the translation length $\tau(W_{n})$ and $d(x_{2n}, x_{0})$ differ by at most $E_{0}$.
\end{enumerate}

Now, for each $\w \in \mathcal{E}$ we define \[
\begin{aligned}
D^{+} &:= \sum_{l=0}^{n} d(x_{2l}, x_{2l+1}), &D^{-} &:= \sum_{l=0}^{n} d(x_{2l+1}, x_{2l}), \\
N(\w) &:= \# \{ i = 1, \ldots, m(n) : v_{i} = \phi_{new} \}.
\end{aligned}
\]
Note that $D^{+}$ and $D^{-}$ are constant on $\mathcal{E}$. Moreover, we have \[\begin{aligned}
 d\big(x_{0}(\w), x_{2m(n)} (\w) \big) &:= \sum_{l=1}^{2m(n)} d(x_{l-1}, x_{l}) + \sum_{l=1}^{2m(n)-1} d(x_{0}, x_{l+1})_{x_{l}}\\
&\in \left( D^{+} + \sum_{l=1}^{m(n)} d(x_{2l-1}, x_{2l}) - 2m(n) \cdot E_{0}, D^{+} + \sum_{l=1}^{m(n)} d(x_{2l-1}, x_{2l}) + 2m(n) \cdot E_{0} \right).
\end{aligned}
\]
Also, recall that $d(o, v_{l} o)$ and $d(x_{2l-1}, x_{2l})$ differ by at most $2K_{0} M_{0}$, and that  \[
\begin{aligned}
\sum_{l=1}^{m(n)} d(o, v_{l} o) = NL_{1}^{+} + (m(n) - N) L_{2}^{+} \\
\end{aligned}
\]
Finally, $d\big(x_{0}(\w), x_{2m(n)} (\w) \big)$ and $\tau(Z_{n}(\w))$ differ by at most $E_{0}$. Combining these, we conclude that \[
\left|  \tau(Z_{n}(\w))- \big(NL_{1}^{+} + (m(n) - N) L_{2}^{+} \big) \right| \le 2m(n) \cdot E_{0} + 2K_{0} M_{0}m(n) + E_{0}.
\]
Similarly, using the equality \[
\sum_{l=1}^{m(n)} d(v_{l} o, o) = NL_{1}^{-} + (m(n) - N) L_{2}^{-},
\]
we conclude that \[
\left|  \tau(Z_{n}^{-1}(\w))- \big(NL_{1}^{-} + (m(n) - N) L_{2}^{-} \big) \right| \le 2m(n) \cdot E_{0} + 2K_{0} M_{0}m(n) + E_{0}.
\]

This implies that $\tau(Z_{n}(\w)) - \tau(Z_{n}^{-1}(\w))$ falls in to the interval $I_{N}$ of width $(6E_{0} + 4K_{0}M_{0})n'$, centered at $c_{N} := N(L_{1}^{+} - L_{1}^{-}) + (m(n) - N) (L_{2}^{+} - L_{2}^{-})$. Note that \[
c_{N+1} - c_{N} = (L_{1}^{+} - L_{1}^{-}) - (L_{2}^{+} - L_{2}^{-}) > (6E_{0} + 4K_{0} M_{0} + 2)n'.
\]
Hence, the intervals $\{I_{N}\}_{N=0}^{m(n)}$ are disjoint and $2n'$-separated from each other. This implies that only at most one of them can intersect the interval $[-n', n']$. 

Meanwhile, note that \[\begin{aligned}
\max_{k} \Prob \left( N(\w) = k \, \big | \,\w \in \mathcal{E} \right) &= \max_{k} \Prob \left( \sum_{l=1}^{m(n)} B_{l}  = k\right) \quad \left( \begin{array}{c} \textrm{$\{B_{i}\}_{i=1}^{m(n)}$ are i.i.d. }\\ \textrm{Bernoulli RVs with mean $1/2$}\end{array}\right) \\
&= \max_{k}2^{-m(n)} \binom{ m(n)}{k} = O\left(1/\sqrt{m(n)}\right) = O\left(1/\sqrt{n'}\right).
\end{aligned}
\]
From this, we obtain \[
\Prob \left( |\tau(Z_{n}(\w)) - \tau(Z_{n}^{-1}(\w)) | \le n' \, \big| \, \w \in \mathcal{E} \right) = O\left(1/\sqrt{n'}\right).
\]
Since $A_{n}$ is partitioned with such an equivalence class $\mathcal{E}$ and takes up most of $\Omega_{n}$ asymptotically, we conclude that \[
\liminf_{n \rightarrow \infty} \Prob \left( |\tau(Z_{n}(\w)) - \tau(Z_{n}^{-1}(\w)) | \le n'  \right) = O\left(1/\sqrt{n'}\right).
\]
We now send $n'$ to infinity to conclude.
\end{proof}

\section{Outer space} \label{section:outer}

In this section, we collect facts about the outer automorphism group and Outer space. For detailed definitions and theories, we refer the readers to the general exposition by Vogtmann \cite{vogtmann2015outer} or individual papers, e.g. \cite{bestvina1992train}, \cite{francaviglia2011metric}, \cite{francaviglia2012isometry}, \cite{algom-kfir2012asymmetry}, \cite{algom-kfir2011strongly}, \cite{dahmani2018spectral},  \cite{dowdall2018hyperbolic} and \cite{kapovich2022random}.

Let $X$ be the Culler-Vogtmann Outer space $CV_{N}$ of rank $N\ge 3$, which is the space of unit-volume marked metric graphs with fundamental group $F_{N}$. In other words, a point $p \in CV_{N}$ corresponds to the homotopic class of a homotopy equivalence $h : R_{N} \rightarrow \Gamma$, where $R_{N}$ is a fixed rose with $N$ petals and $\Gamma$ is a unit-volume metric graph. The corresponding space without the volume normalization is called the unprojectivized Outer space $cv_{N}$, and there is a projectivization from $cv_{N}$ to $CV_{N}$ by dilation.

Outer space comes equipped with a canonical metric, the Lipschitz distance, which is defined as follows: for two markings $h_{1} : R_{N} \rightarrow \Gamma_{1}$ and $h_{2} : R_{N} \rightarrow \Gamma_{2}$, the distance from $\Gamma_{1}$ to $\Gamma_{2}$ is defined by \[
d_{CV}(\Gamma_{1}, \Gamma_{2}) := \inf \{ \log \operatorname{Lip}(f) : f \sim f_{2} \circ f_{1}^{-1} \},
\]
where $\operatorname{Lip}(f)$ is the (maximal) Lipschitz constant of $f$. \emph{We now make a convention that differs from the traditional one}. Namely, the outer automorphism group $\Out(F_{N})$ of rank $N$ acts on $CV_{N}$ by changing the basis of the marking \emph{with the inverses}: given $\phi \in \Out(F_{N})$ and $h : R_{N} \rightarrow \Gamma$ representing a point of $CV_{N}$, $\phi$ moves $h$ to $h \circ \phi^{-1} : F_{N} \stackrel{\phi^{-1}}{\rightarrow} F_{N} \stackrel{h}{\rightarrow} \Gamma$. This is a left action by isometries. We denote action by $X \ni h \mapsto \phi \cdot h \in X$.

It is known that the Lipschitz distance is asymmetric \cite{francaviglia2011metric} and not uniquely geodesic. However, distances among $\epsilon$-thick points (i.e., those with systole at least $\epsilon$) have the coarse symmetry: there exists a constant $C=C(\epsilon) < +\infty$ such that for any $\epsilon$-thick points $x$ and $y$, one has $d(x, y) \le Cd(y, x)$ \cite{algom-kfir2012asymmetry}. In particular, distances among the translates of the reference point $o$ by $\Out(F_{N})$ satisfy the coarse symmetry.

Just as Teichm{\"u}ller space $\T(\Sigma)$ is accompanied by the curve complex $\mathcal{C}(\Sigma)$ and the coarse projection $\pi^{\mathcal{C}} : \T(\Sigma) \rightarrow \mathcal{C}(\Sigma)$, $CV_{N}$ is accompanied by the complex of free factors $\mathcal{FF}_{N}$ and the coarse projection $\pi^{\mathcal{FF}} : CV_{N} \rightarrow \mathcal{FF}_{N}$. This projection is coarsely $\Out(F_{N})$-equivariant and coarsely Lipschitz. Moreover, geodesics in $CV_{N}$ projects to $K$-unparametrized bi-quasigeodesics for some uniform $K > 0$ \cite[Proposition 9.2]{bestvina2014hyperbolicity}. 

Outer space also accomodates lots of BGIP isometries. We say that an outer automorphism $\phi \in \Out(F_{N})$ is \emph{reducible} if there exists a free product decomposition $F_{N} = C_{1} \ast \cdots \ast C_{k} \ast C_{k+1}$, with $k \ge 1$ and $C_{i} \neq \{e\}$, such that $\phi$ permutes the conjugacy classes of $C_{1}, \ldots, C_{k}$. If not, we say that $\phi$ is \emph{irreducible}. We also say that $\phi$ is \emph{fully ireducible (or iwip)} if no power of $\phi$ is reducible, or equivalently, no power of $\phi$ preserves the conjugacy class of any proper free factor of $F_{N}$. We also say that $\phi$ is \emph{atoroidal (or hyperbolic)} if no power of $\phi$ fixes any nontrivial conjugacy class in $F_{N}$. When $\phi$ is fully irreducible, it is non-atoroidal if and only if it is \emph{geometric}, i.e., induced by a pseudo-Anosov $\varphi : \Sigma \rightarrow \Sigma$ on a compact surface $\Sigma$ with one boundary component, via an identification of $F_{N}$ with $\pi(\Sigma)$. Bestvina and Feighn proved in \cite{bestvina2014hyperbolicity} that $\phi \in \Out(F_{N})$ is fully irreducible if and only if it acts on $\mathcal{FF}_{N}$ loxodromically.

We say that a subgroup $G \le Out(F_{N})$ is \emph{non-elementary} if it acts on $\mathcal{FF}_{N}$ in a non-elementary way, or equivalently, contains two fully irreducibles with mutually distinct attracting/repelling trees. It is known that if $G \le Out(F_{N})$ does not fix any finite subset of $\mathcal{FF}_{N} \cup \partial \mathcal{FF}_{N}$, or equivalently, if it is not virtually cyclic nor virtually fixes the conjugacy class of a proper free factor of $F_{N}$, then $G$ is non-elementary \cite{horbez2016short}. Since $\pi^{\mathcal{FF}}$ is coarsely Lipschitz, the independence of two fully irreducibles in $\mathcal{FF}_{N}$ is lifted to the independence in $CV_{N}$. 

We refer the readers to \cite{bestvina1992train}, \cite{algom-kfir2010lipschitz}, \cite{bestvina2014hyperbolicity} and \cite{algom-kfir2019stable} for the precise definition of a train-track representative of an outer automorphism. Roughly speaking, a train-track representative of $\phi$ is a self-map $f : \Gamma \rightarrow \Gamma$ in the free homotopy class of $\phi$ on a simplicial graph $\Gamma$ that sends vertices to vertices, restricts to an immersion on each edge of $\Gamma$, and sends edges to immersed segments after iterations. It is due to Bestvina and Handel \cite{bestvina1992train} that every irreducible outer automorphism admits a train-track representative, although it may not be unique.

Given such a structure, one can endow $\Gamma$ with a metric such that $f$ stretches each edge of $\Gamma$ by the same constant $\lambda > 1$, which is called the \emph{expansion factor} of $f$. This expansion factor is uniquely determined by the choice of $\phi$ and does not depend on the choice of $f$. Moreover, in view of Skora's interpretation of Stallings fold decompositions, one obtains a continuous path on $cv_{N}$ from $\Gamma$ to $\Gamma \circ \phi$ by folding a single illegal turn at each time (cf. \cite{algom-kfir2019stable}). This descends to a geodesic segment of length $\log \lambda$ (after a reparametrization) and the concatenation of its translates by powers of $\phi$ becomes a bi-infinite, $\phi$-periodic geodesic. We call this a \emph{(optimal) folding axis} of $\phi$. Algom-Kfir observed the following: 

\begin{thm}[\cite{algom-kfir2011strongly}]\label{thm:algomKfir}
Folding axes of fully irreducible outer automorphisms are strongly contracting.
\end{thm}

Meanwhile, we need BGIP instead of the strongly contracting property in our setting, and the author does not know a way to promote the latter to the former. Meanwhile, I. Kapovich, Maher, Pfaff and Taylor observed the following version of BGIP in Outer space. This requires the notion of greedy folding path, whose accurate definition can be found in \cite{francaviglia2011metric}, \cite{bestvina2014hyperbolicity} and \cite{dahmani2018spectral}. In short, a greedy folding path $\gamma : I \rightarrow cv_{N}$ is obtained by folding every illegal turns at each time with speed 1, where the illegal turn structures at different forward times are identical and define a well-defined illegal turn structure. This descends to a geodesic on $CV_{N}$, and we have the following theorem: 

\begin{thm}[{\cite[Theorem 7.8]{kapovich2022random}}] \label{thm:kapovich}
Let $\phi \in \Out(F_{N})$ be a fully irreducible outer automorphism. Suppose that $\gamma$ is a bi-infinite, $\phi$-periodic greedy folding path. Then there exist $C>0$ such that the following holds.

Let $x, y \in X$ be points such that $d^{sym}(\pi_{\gamma}(x), \pi_{\gamma}(y)) \ge C$, and satisfy $d^{sym}(\pi_{\gamma}(x)) = \gamma(t_{1})$, $ d^{sym}(\pi_{\gamma}(y)) = \gamma(t_{2})$ for some $t_{1} < t_{2}$. Then any geodesic $[x, y]$ between them contains a subsegment $[z_{1}, z_{2}]$ such that \[
d^{sym}(z_{1}, \pi_{\gamma}(x)) <C,\quad d^{sym}(z_{2}, \pi_{\gamma}(y)) < C.
\]
\end{thm}

This uni-directional version of BGIP is designed for outer automorphisms that have an invariant(=periodic) greedy folding line. It seems not shown that all fully irreducibles have such a line. (The author thanks Sam Taylor for pointing this out.)  Nonetheless, by adapting Dowdall-Taylor's idea and Kapovich-Maher-Pfaff-Taylor's proof of Theorem \ref{thm:kapovich}, we can obtain the following result. This proof was kindly informed by Sam Taylor.

\iwipBGIP*

\begin{proof}
Before we begin, we recall the following facts regarding a geodesic $\delta$-hyperbolic space $Y$.
\begin{enumerate}
\item (Morse property) A $K$-quasigeodesic and a geodesic with the same endpoints are within Hausdorff distance $K_{2} = K_{2}(K, \delta)$.
\item The closest point projections of a point $y \in Y$ onto a $K$-quasigeodesic and a geodesic on $Y$ with the same endpoints are within distance $K_{3} = K_{3}(K, \delta)$.
\item If the projections of $x, y \in Y$ onto  a $K$-quasigeodesic $\gamma$ contain $\gamma(s)$ and $\gamma(t)$, respectively, and if $d(\gamma(s), \gamma(t)) > K_{4} = K_{4}(K, \delta)$, then $[x, y]$ and $[x, \gamma(s)] \cup \gamma|_{[s, t]}\cup [\gamma(t), y]$ are within Hausdoff distance $K_{4}$.
\item If two $K$-quasigeodesics $\gamma$, $\gamma'$ are within Hausdorff distant $K$ and the distance between starting points is at most $K$, then $\gamma'$ crosses $\gamma$ up to a constant $K_{5} = K_{5}(K, \delta)$, i.e., $\gamma$ and $\gamma' \circ \rho$ are $K_{5}$-fellow traveling for some non-decreasing reparametrization $\rho$.
\end{enumerate}

Let $o \in CV_{N}$ be an arbitrary basepoint. Let $T^{+}$, $T^{-}$ be the attracting and repelling trees of $\varphi$, respectively. There exist optimal greedy folding lines $\gamma^{\pm} : \mathbb{R} \rightarrow CV_{N}$ such that \begin{equation}\label{eqn:greedyFold}
\lim_{t \rightarrow +\infty} \gamma^{\pm}(t) = T^{\pm}, \quad \lim_{t\rightarrow -\infty} \gamma^{\pm}(t) = T^{\mp}
\end{equation}
(\cite{bestvina2015boundary}, Lemma 6.7 and Lemma 7.3). 
Since $\{\varphi_{i} o\}_{i}$ is a quasigeodesic whose endpoints agree with $\gamma^{+}$, Theorem 4.1 of \cite{dowdall2018hyperbolic} asserts that $d_{H}(\{\varphi^{i} o\}_{i}, \gamma^{+}) < K_{1}$ and $\pi^{\mathcal{FF}}(\gamma^{+})$ is a $K_{1}$-quasigeodesic for some $K_{1}$. Similarly, by comparing $\{\varphi_{-i} o\}_{i}$ and $\gamma^{-}$, we deduce that $d_{H}(\{\varphi^{i} o\}_{i}, \gamma^{-})<K_{1}$ and $\pi^{\mathcal{FF}}(\gamma^{-})$ is a $K_{1}$-quasigeodesic. Also, $\gamma^{\pm}$ are uniformly thick.

Recall our notation: $\pi_{\gamma^{\pm}}$ denotes the nearest point projection onto $\gamma^{\pm}$. For now, we denote by $\pi_{\gamma^{+}}^{sym}(x)$ the nearest point projection of $x \in CV_{N}$ onto $\gamma^{+}$ \emph{with respect to $d^{sym}$}. Similarly, we denote by $\pi_{\gamma^{-}}^{sym}(\cdot)$ the $d^{sym}$-nearest point projection onto $\gamma^{-}$.

Let us now take $x_{i}^{+} \in \pi_{\gamma^{+}}^{sym}(\varphi^{i} o)$ and $x_{i}^{-} \in \pi_{\gamma^{-}}^{sym}(\varphi^{i} o)$ for each $i$. We recall the following result of Dahmani and Horbez (\cite[Proposition 5.17, Corollary 5.22]{dahmani2018spectral}; see also Section 7 of \cite{kapovich2022random}): there exist $B, D>0$ such that $\gamma^{\pm}$ are $(B, D)$-contracting at $x_{i}^{\pm}$'s (with a suitable crossing constant $\kappa$). In other words, a geodesic $\eta$ on $CV_{N}$ has the projection $\pi^{\mathcal{FF}} \circ \eta$ that $\kappa$-crosses up the $B$-long subsegment of $\pi^{\mathcal{FF}} \gamma^{\pm}$ that begins from $\pi^{\mathcal{FF}}(x_{i}^{\pm})$, then $\eta$ has a point $p$ such that $d(p, x_{i}^{\pm}) \le D$. Since $\gamma^{\pm}$ are thick, $d(x_{i}^{\pm}, p)$ is (uniformly) proportional to $d(p, x_{i}^{\pm})$ and $\eta$ intersects a uniform $d^{sym}$-neighborhood of $\gamma^{\pm}$ in such a case.

We now observe that $\pi^{\mathcal{FF}} \pi_{\gamma^{+}}, \pi^{\mathcal{FF}}\pi_{\gamma^{-}}$ and $\pi_{\pi^{\mathcal{FF}}(\{\varphi^{i} o\}_{i})} \circ \pi^{\mathcal{FF}}$ are coarsely equivalent. First, \cite[Lemma 4.11]{dowdall2018hyperbolic} asserts that $\pi_{\gamma^{\pm}}$ and $Pr_{\gamma^{\pm}}$ are equivalent, where $Pr$ stands for the Bestvina-Feighn left projection. Then \cite[Lemma 4.2]{dowdall2018hyperbolic} asserts that $\pi^{\mathcal{FF}} Pr_{\gamma^{\pm}}$ and $\pi_{\pi^{\mathcal{FF}}(\gamma^{\pm})} \circ \pi^{\mathcal{FF}} $ are equivalent. These are then equivalent to $\pi_{\pi^{\mathcal{FF}}(\{\varphi^{i} o\}_{i})} \circ \pi^{\mathcal{FF}}$, since $\pi^{\mathcal{FF}}(\gamma^{\pm})$ and $\pi^{\mathcal{FF}}(\{\varphi^{i} o\}_{i})$ are equivalent quasi-geodesics on the Gromov hyperbolic space $\mathcal{FF}$.

We now lift these projections: we claim that $\pi_{\gamma^{+}}$, $\pi_{\gamma^{-}}$ and $\pi_{\{\varphi^{i} o\}_{i}}$ are equivalent. First, suppose that $\pi_{\gamma^{+}}(x)$ and $\pi_{\gamma^{-}}(x)$ are far from each other for some $x \in X$. Since $\gamma^{+}$, $\gamma^{-}$, $\{\varphi^{i}o\}_{i}$ are close to each other, we may take $\varphi^{i} o$ and $\varphi^{j} o$ near $\pi_{\gamma^{+}}(x)$ and $\pi_{\gamma^{-}}(x)$, respectively, and conclude that $|i-j|$ is large. This implies that $\pi^{\mathcal{FF}}(\varphi^{i} o)$ and $\pi^{\mathcal{FF}}(\varphi^{j} o)$ are also far from each other (since $\varphi$ is loxodromic on $CV_{N}$), and consequently $\pi^{\mathcal{FF}}(\pi_{\gamma^{+}}(x))$, $\pi^{\mathcal{FF}}(\pi_{\gamma^{-}}(x))$ are far from each other. ($\ast$) Since  $\pi^{\mathcal{FF}}\pi_{\gamma^{+}}$ and $\pi^{\mathcal{FF}}\pi_{\gamma^{-}}$ are equivalent, this cannot happen. For this reason, $\pi_{\gamma^{+}}$ and $\pi_{\gamma^{-}}$ are equivalent. 

Now suppose that $\pi_{\{\varphi^{i} o\}_{i}}(x)$ and $\pi_{\gamma^{\pm}}(x)$ are far from each other for some $x \in X$. We take $\varphi^{j} o \in \pi_{\{\varphi^{i} o\}_{i}}(x)$ and $\varphi^{j'} o$ near $\pi_{\gamma^{\pm}}(x)$ and conclude that $|j' - j|$ is large. If $j \gg j'$, then $\pi^{\mathcal{FF}}([x, \varphi^{j} o])$ is a quasigeodesic whose endpoints project onto $\pi^{\mathcal{FF}}(\{\varphi^{i} o\}_{i})$ near $\pi^{\mathcal{FF}} \varphi^{j'}o$ and $\pi^{\mathcal{FF}} \varphi^{j} o$, respectively. Since $j' - j$ is large enough, this quasigeodesic crosses up long enough subsegments of $\pi^{\mathcal{FF}}(\{\varphi^{i} o\}_{i})$ and $\pi^{\mathcal{FF}}(\gamma^{+})$ that begin at $\pi^{\mathcal{FF}}(\varphi^{j'} o)$ and $\pi^{\mathcal{FF}}(x_{j'})$, respectively. Using the $(B, D)$-contraction at $x_{j'}^{+}$ of $\gamma^{+}$, we conclude that $[x, \varphi^{j} o]$ contains a point $p$ nearby $x_{j'}^{+}$, which makes $d(x, \varphi^{j'} o)$ shorter than $d(x, \varphi^{j} o)$ and leads to a contradiction. Similar contradiction occurs when $j' \gg j$,  due to the contracting property of $\gamma^{-}$ at $x_{i}^{-}$'s. Hence, $\pi_{\{\varphi^{i} o\}_{i}}(x)$ and $\pi_{\gamma^{\pm}}(x)$ are equivalent.

Now, if a geodesic $\eta$ on $CV_{N}$ has a large nearest point projection on $\{\varphi^{i} o\}_{i}$, then it also has large projections on $\gamma^{\pm}$. This leads to large $\pi^{\mathcal{FF}}(\pi_{\gamma^{\pm}}(\eta))$, because $\pi^{\mathcal{FF}}$ restricted on $\gamma^{\pm}$ is a QI-embedding (again due to \cite[Theorem 4.1]{dowdall2018hyperbolic}). When $\pi^{\mathcal{FF}}(\pi_{\gamma^{\pm}}(\eta))$ progresses in the forward direction with respect to $\{\varphi_{i} o\}_{i}$, then we employ the contracting property of $\gamma^{+}$ to conclude that $d^{sym}(\eta, \gamma^{+})$ is uniformly bounded. If it progresses in the backward direction, then we employ the contracting property of $\gamma^{-}$ to conclude.
\end{proof}

We now explain more details on the classification of fully irreducible outer automorphisms. Coulbois and Hilion classified fully irreducibles into the following mutually distinct categories in \cite{coulbois2012botany}: \begin{enumerate}
\item \emph{geometric} automorphisms that have geometric attracting and repelling trees,
\item \emph{parageometric} automorphisms that have geometric attracting tree and non-geometric attracting tree,
\item inverses of parageometric automorphisms that have geometric repelling tree and non-geometric attracting tree, and 
\item \emph{pseudo-Levitt} automorphisms that have non-geometric attracting and repelling trees.
\end{enumerate}
Sometimes, automorphisms of category (3) and (4) are together called \emph{ageometric} automorphisms. 

In Theorem \ref{thm:mismatch} we are concerned with fully irreducibles whose expansion factors differ from that of their inverses. Examples of such automorphisms are given in \cite{handel2007parageometric}: every parageometric fully irreducible automorphism has expansion factor larger than the expansion factor of its inverse. Meanwhile, every geometric fully irreducible automorphism and its inverse have the same expansion factor, due to an analogous fact for pseudo-Anosov mapping classes. For pseudo-Levitt automorphisms, both situations can happen (\cite{handel2007parageometric}, \cite{coulbois2012botany}).

As mentioned before, fully irreducibles in $\Out(F_{2})$ are always geometric. In contrast, \cite{jager2008free} provides an example of parageometric automorphism for each $N \ge 3$. Hence, given an admissible probability measure $\mu$ on $\Out(F_{N})$ for $N\ge 3$, there exists $n > 0$ such that $\supp \mu^{\ast n}$ contains a parageometric automorphism $\varphi$, while there exists $m > 0$ such that $\supp \mu^{\ast m}$ contains identity. Then $\supp \mu^{\ast mn}$ contains $\varphi^{m}$ and $id$, where \[
\tau(\varphi^{m}) - \tau(\varphi^{-m}) > 0 = \tau(id) - \tau(id).
\]
For this reason, every admissible probability measure on $\Out(F_{N})$ is asymptotically asymmetric.

When $N \ge 3$, if a non-elementary random walk $(Z_{n})_{n>0}$ on $\Out(F_{N})$ has bounded support that generates a semigroup containing a principal fully irreducible and an inverse of a principal fully irreducible, then the probability that $Z_{n}$ is pseudo-Levitt tends to 1 as $n \rightarrow \infty$ \cite[Theorem A]{kapovich2022random}. Hence, the above property of parageometric automorphisms does not imply that a generic automorphism and its inverse have different expansion factors.

\section{Limit laws for Outer space}\label{section:limitOuter}

By applying the limit laws in Section \ref{section:limit} to the outer automorphism group (with Theorem \ref{thm:iwipBGIP} in hand), we obtain the following results. We first recover the escape to infinity of random walks on $\Out(F_{N})$ due to Horbez \cite{horbez2016horo}, and weaken the moment condition the SLLN for translation length in \cite[Theorem 1.4]{dahmani2018spectral}.

\begin{thm}\label{thm:OuterLLN}
Let $(X, d)$ be the Culler-Vogtmann Outer space $CV_{N}$ of rank $N \ge 3$ with the Lipschitz metric, and let $(Z_{n})_{n}$ be the random walk generated by a non-elementary probability measure $\mu$ on $\Out(F_{N})$. Then there exists a strictly positive quantity $\lambda(\mu) \in (0, +\infty]$, called the \emph{drift} of $\mu$, such that  \[
\lambda(\mu) := \lim_{n \rightarrow \infty} \frac{1}{n} d(o, Z_{n} o) \quad \textrm{almost surely.}
\]
Moreover, for each $0 < L < \lambda(\mu)$, there exists $K > 0$ such that \[
\Prob \Big( \textrm{$Z_{n}$ has BGIP and $\tau(Z_{n}) \ge Ln$} \Big) \ge 1 - K e^{-n/K}.
\]
\end{thm}

Next, we recover Horbez's CLT on $\Out(F_{N})$ \cite{horbez2018clt} and establish its converse: 

\begin{thm}\label{thm:OuterCLTStrong}
Let $(X, d)$ be the Culler-Vogtmann Outer space $CV_{N}$ of rank $N \ge 3$ with the Lipschitz metric, and let $(Z_{n})_{n}$ be the random walk generated by a non-elementary probability measure $\mu$ on $\Out(F_{N})$. \begin{enumerate}
\item Suppose that $\mu$ has finite second moment.  Then the following limit (called the \emph{asymptotic variance of $\mu$}) exists: \[
\sigma^{2} (\mu) := \lim_{n \rightarrow \infty} \frac{1}{n} Var[d(o, Z_{n} o)],
\]
and the random variables $\frac{1}{\sqrt{n}} [d(o, Z_{n}o) - \lambda(\mu) n]$ and $\frac{1}{\sqrt{n}} [\tau(Z_{n}) - \lambda(\mu) n]$ converge in law to the Gaussian law $\mathcal{N}(0, \sigma(\mu))$ with zero mean and variance $\sigma^{2}(\mu)$. Furthermore, we have  \[
\limsup_{n \rightarrow \infty} \frac{d(o, Z_{n} o) - \lambda(\mu) n}{\sqrt{2n \log \log n}} =
\limsup_{n \rightarrow \infty} \frac{\tau(Z_{n})- \lambda(\mu) n}{\sqrt{2n \log \log n}} = \sigma(\mu).
\] 
\item Suppose that $\mu$ has infinite second moment. Then neither $(\frac{1}{\sqrt{n}} d(o, Z_{n}o) - c_{n})_{n>0}$ nor $(\frac{1}{\sqrt{n}} \tau(Z_{n}) - c_{n})_{n > 0}$ converge in law for any choice of the sequence $(c_{n})_{n>0}$.
\end{enumerate}
\end{thm}

Next, we discuss the geodesic tracking of random walks, which is a generalization of \cite[Theorem 7.2]{kaimanovich2000hyp}, \cite[Theorem 5]{tiozzo2015sublinear} and \cite[Theorem 1.3]{maher2018random} to Outer space.

\begin{thm}\label{thm:OuterTracking }
Let $(X, d)$ be the Culler-Vogtmann Outer space $CV_{N}$ of rank $N \ge 3$ with the Lipschitz metric, and let $(Z_{n})_{n}$ be the random walk generated by a non-elementary probability measure $\mu$ on $\Out(F_{N})$. \begin{enumerate}
\item Suppose that  $\mu$ has finite $p$-th moment for some $p > 0$. Then for almost every sample path $(Z_{n}(\w))_{n \ge 0}$, there exists a quasigeodesic $\gamma$ such that \[
\lim_{n \rightarrow\infty} \frac{1}{n^{1/2p}} d^{sym}(Z_{n} o, \gamma) = 0.
\]
\item Suppose that $\mu$ and $\check{\mu}$ have finite exponential moment. Then there exists $K<+\infty$ such that for almost every sample path $(Z_{n}(\w)_{n\ge 0}$, there exists a quasigeodesic $\gamma$ satisfying \[
\lim_{n \rightarrow \infty}  \frac{1}{\log n} d^{sym}(Z_{n} o, \gamma) < K.
\]
\end{enumerate}
\end{thm}

We also discuss the large deviation principle for the RVs $\{d(o, Z_{n} o)/n\}_{n >0}$, by combining the strategy of \cite{boulanger2022large} and \cite{gouezel2022exponential}.

\begin{thm}\label{thm:OuterLDPStrong}
Let $(X, d)$ be the Culler-Vogtmann Outer space $CV_{N}$ of rank $N \ge 3$ with the Lipschitz metric, and let $(Z_{n})_{n}$ be the random walk generated by a non-elementary probability measure $\mu$ on $\Out(F_{N})$. Let $\lambda(\mu) = \lim_{n} \frac{1}{n} \E[d(o, Z_{n} o)]$ be the drift of $\mu$. Then for each $0 < L <\lambda(\mu)$, the probability $\Prob(d(o, Z_{n} o) \le Ln)$ decays exponentially as $n$ tends to infinity.
\end{thm}

\begin{thm} \label{cor:OuterLDPMod}
Let $(X, d)$ be the Culler-Vogtmann Outer space $CV_{N}$ of rank $N \ge 3$ with the Lipschitz metric, and let $(Z_{n})_{n}$ be the random walk generated by a non-elementary probability measure $\mu$ on $\Out(F_{N})$. Then there exists a proper convex function $I : \mathbb{R} \rightarrow [0, +\infty]$, vanishing only at the drift $\lambda(\mu)$, such that  \[\begin{aligned}
- \inf_{x \in \operatorname{int}(E)} I(x) &\le \liminf_{n \rightarrow \infty} \frac{1}{n} \log \Prob\left( \frac{1}{n}d(id, Z_{n}) \in E\right), \\
 -\inf_{x \in \bar{E}} I(x) & \ge\limsup_{n \rightarrow \infty} \frac{1}{n} \log \Prob\left(\frac{1}{n} d(id, Z_{n}) \in E\right) 
\end{aligned}
\]
holds for every measurable set $E \subseteq \mathbb{R}$.
\end{thm}

\begin{thm}\label{thm:OuterDiscrepancy}
Let $(X, d)$ be the Culler-Vogtmann Outer space $CV_{N}$ of rank $N \ge 3$ with the Lipschitz metric, and let $(Z_{n})_{n}$ be the random walk generated by a non-elementary probability measure $\mu$ on $\Out(F_{N})$. If $\mu$ has finite $p$-th moment for some $p > 0$, then \[
\lim_{n \rightarrow \infty} \frac{1}{n^{1/2p}} [d(o, Z_{n} o)- \tau(Z_{n}) ]= 0 \quad \textrm{a.s.}
\]
Furthermore, if $\mu$ has finite first moment, then there exists $K>0$ such that \[
\lim_{n \rightarrow \infty} \frac{1}{\log n} [d(o, Z_{n} o)- \tau(Z_{n}) ]< K \quad \textrm{a.s.}
\]
\end{thm}

The following is an effective version of \cite[Theorem 1.4]{taylor2016random} regarding genericity of q.i. embedded subgroups of $\Out(F_{N})$.

\begin{thm}\label{thm:qiOuter} 
Let $(X, d)$ be the Culler-Vogtmann Outer space $CV_{N}$ of rank $N \ge 3$ with the Lipschitz metric and let  $(Z_{n}^{(1)}, \ldots, Z_{n}^{(k)})_{n \ge 0}$ be $k$ independent random walks generated by a non-elementary measure $\mu$ on $\Out(F_{N})$. Then there exists $K>0$ such that \[
\Prob \left[ \langle Z_{n}^{(1)}, \ldots, Z_{n}^{(k)} \rangle \,\,\textrm{is q.i. embedded into a quasi-convex subset of $X$} \right] \ge 1- K e^{-n/K}.
\]
\end{thm}

By using the technique in \cite{choi2021generic}, \cite{choi2022random2}, we obtain an analogue in the counting problem.

\begin{thm}\label{thm:qiCountOuter}
Let $(X, d)$ be the Culler-Vogtmann Outer space $CV_{N}$ of rank $N \ge 3$ with the Lipschitz metric.  Then for each $k > 0$, there exists a finite generating set $S$ of $\Out(F_{N})$ such that \[
\frac{\#\left\{(g_{1}, \ldots, g_{k}) \in \big(B_{S}(n) \big)^{k} :\begin{array}{c}  \langle g_{1}, \ldots, g_{k} \rangle \,\, \textrm{is q.i. embedded into} \\ \textrm{a quasi-convex subset of $X$} \end{array} \right\} } {\big(\# B_{S}(n)\big)^{k}}
\]
converges to 1 exponentially fast.
\end{thm}

Since any fully irreducible has BGIP, we can apply Lemma \ref{lem:coarseEquiv} to have the following version of Proposition \ref{prop:stability}:

\begin{prop}\label{prop:OuterStability}
Let $(X, d)$ be the Culler-Vogtmann Outer space $CV_{N}$ of rank $N \ge 3$ with the Lipschitz metric and let $(Z_{n})_{n>0}$ be the random walk generated by a non-elementary probability measure $\mu$ on $\Out(F_{N})$. Then for each fully irreducible $g$ generated by the support of $\mu$, there exists $K>0$ such that the following holds. 

For each positive integer $M$, there exists $\epsilon>0$ such that the following holds outside a set of exponentially decaying probability: for a random path $(Z_{n})_{n>0}$, there exist $0 <i(1) < \ldots < i(\epsilon n) < n$ and there exist disjoint subsegments $\eta_{1}, \ldots, \eta_{\epsilon n}$ of $[o, Z_{n} o]$, from closest to farthest from $o$, such that $\eta_{k}$ and $[Z_{i(k)}, Z_{i(k)} g^{M} o]$ are $K$-fellow traveling for $k=1, \ldots, \epsilon n$.
\end{prop}

In \cite{kapovich2022random}, the author observed the following stability of triangular fully irreducible automorphisms. For the definition of principal and triangular fully irreducibles, refer to \cite[Section 5.3]{kapovich2022random}; for us, it it sufficient to know that principal and triangular fully irreducibles are ageometric.

\begin{lem}[{\cite[Corollary 5.14]{kapovich2022random}}] \label{lem:stableLine}
Let $N \ge 3$. 
Let $\varphi$ be a fully irreducible outer automorphism in $\Out(F_{N})$ with the lone axis $\gamma$ on $CV_{N}$ (this is satisfied when $\varphi$ is a principal fully irreducible). Then for each $\epsilon > 0$ there exists $R \ge 1$ such that, if $\psi$ is a fully irreducible automorphism with an axis $\gamma'$ such that some $R$-long subsegments of $\gamma$ and $\gamma'$ are $\epsilon$-fellow traveling, then $\psi$ is a triangular fully irreducible.
\end{lem}

By combining Lemma \ref{lem:stableLine} with Proposition \ref{prop:OuterStability}, we recover version of Kapovich-Maher-Pfaff-Taylor's result with weaker moment condition. 

\begin{thm}[cf. {\cite[Theorem A]{kapovich2022random}}] \label{thm:ageom}
Let $N \ge 3$ and let $\mu$ be a non-elementary probability measure on $\Out(F_{N})$ such that the subsemigroup $\llangle \supp \mu \rrangle$ generated by the support of $\mu$ contains the inverse of a principal fully irreducible automorphism. Then outside a set of exponentially decaying probability, $Z_{n}$ is an ageometric triangular fully irreducible outer automorphism.
\end{thm}

Finally, Theorem \ref{thm:mismatch} reads as follows:

\begin{thm}\label{thm:OuterMismatch}
Let $(X, d)$ be the Culler-Vogtmann Outer space $CV_{N}$ of rank $N \ge 3$ with the Lipschitz metric and  let $(Z_{n})_{n>0}$ be the random walk generated by a non-elementary, asympotically asymmetric probability measure $\mu$ on $\Out(F_{N})$ (for example, $\mu$ is an  admissible probability measure). Then for any $K>0$, we have \[
\lim_{n \rightarrow \infty} \Prob \Big( \w :\big| \tau(Z_{n}) - \tau(Z_{n}^{-1}) \big| < K \Big) = 0.
\]
\end{thm}

\medskip
\bibliographystyle{alpha}
\bibliography{Combined_v3}

\end{document}